\documentclass[11pt]{amsart}
\usepackage[T1]{fontenc}
\usepackage[utf8]{inputenc}
\usepackage{lmodern}
\usepackage[margin=1in]{geometry}
\usepackage{amsmath,amssymb,mathtools,booktabs,array,longtable, comment}
\usepackage{microtype}
\usepackage{xcolor}
\usepackage[colorlinks=true,linkcolor=blue!90!black,citecolor=blue!90!black,urlcolor=blue!90!black]{hyperref}
\newtheorem{theorem}{Theorem}[section]
\newtheorem{proposition}[theorem]{Proposition}
\newtheorem{lemma}[theorem]{Lemma}
\newtheorem{corollary}[theorem]{Corollary}
\theoremstyle{definition}

\newtheorem{example}[theorem]{Example}
\theoremstyle{remark}
\newtheorem{remark}[theorem]{Remark}
\numberwithin{equation}{section}
\newcommand{\Z}{\mathbb Z}
\newcommand{\Q}{\mathbb Q}
\newcommand{\F}{\mathbb F}
\newcommand{\Gm}{\mathbb G_m}
\newcommand{\Mone}{\mathcal M_1}
\newcommand{\A}{\mathcal A}
\newcommand{\D}{\mathbb D}
\newcommand{\BT}{\mathrm{BT}}
\newcommand{\Dieu}{\mathsf{Dieu}}
\newcommand{\Isoc}{\mathsf{Isoc}}
\newcommand{\wt}{\mathsf W}

\newcommand{\DM}{\mathrm{DM}}
\newcommand{\op}{\mathrm{op}}
\newcommand{\id}{\mathrm{id}}
\newcommand{\cris}{\mathrm{crys}}
\DeclareMathOperator{\Hom}{Hom}
\DeclareMathOperator{\End}{End}
\DeclareMathOperator{\Ext}{Ext}
\DeclareMathOperator{\Gal}{Gal}
\DeclareMathOperator{\Lie}{Lie}
\DeclareMathOperator{\gr}{gr}
\DeclareMathOperator{\rk}{rk}
\DeclareMathOperator{\Res}{Res}

\DeclareMathOperator{\im}{im}

\DeclareMathOperator{\Tot}{Tot}
\DeclareMathOperator{\Kar}{Kar}
\DeclareMathOperator{\thick}{thick}
\DeclareMathOperator{\ord}{ord}
\title[Crystalline Realizations of $1$-Motives]{Crystalline Realizations of $1$-Motives: Integral Structure and Classification via Linear Frobenius}
\author{Mohammadreza Mohajer}
\address{Department of Mathematics, Physics and Geology, Cape Breton University, 1250 Grand Lake Road, Sydney, NS B1P 6L2, Canada}
\email{Mohammadreza\_Mohajer@cbu.ca}
\email{mohammadreza.mohajer96@gmail.com}
\hypersetup{pdftitle={Crystalline Realizations of \texorpdfstring{$1$}{1}-Motives: Integral Structure and Classification via Linear Frobenius},pdfauthor={Mohammadreza Mohajer}}

\begin{document}
\begin{abstract}
We study morphisms and integral weight splittings in the contravariant
Dieudonn\'e realization of Deligne $1$-motives over a finite field.
Combining the classical Tate theorem with a finite-level divisibility
criterion, we prove integral full faithfulness after $p$-adic completion
of morphism groups. We identify the geometric extension classes killed
by the Barsotti--Tate realization and determine the exact common
$p$-power denominator of the weight projectors from the extension maps
of a $1$-motive and its Cartier dual. The same invariant measures the
integral splitting defect. Frobenius resultants and reduced resultants
give uniform bounds, whose sharpness is illustrated by explicit
elliptic examples. Semilinear descent expresses crystalline morphisms
through linear Frobenius intertwiners; eigenvalue ratios then determine
morphism ranks over finite extensions and the minimal common field
of definition. We characterize the rational essential image using
cyclotomic factors and Honda--Tate multiplicities, and obtain derived
realization equivalences for motives generated by curves after scalar
extension and idempotent completion.
Applications include quaternionic descent and divisor correspondences between smooth projective curves.
\end{abstract}

\maketitle
\setcounter{tocdepth}{1}
\section{Introduction}
Crystalline Frobenius separates the three weights of a Deligne
$1$-motive over a finite field after inverting $p$. The integral
Dieudonn\'e lattice, however, records how these weight components
are joined. We determine the geometric information retained by this
integral structure, compute the exact denominators needed to split
the weights, and describe how Frobenius controls morphisms under
scalar extension and extension of the ground field.

Let $k=\F_q$, where $q=p^f$, and put $W=W(k)$ and $K_0=W[1/p]$.
Write $\sigma$ for the Witt Frobenius, so that
$K_0^\sigma=\Q_p$ and $\sigma^f=1$. A Deligne $1$-motive is a complex
$M=[L\xrightarrow{u}G]$ in degrees $-1,0$, where $L$ is a lattice
and $G$ is a semi-abelian variety \cite{Deligne}. Its weight components
are a lattice, an abelian variety, and a torus. We use the
contravariant realization
\[
D(M)=\D(M[p^\infty])(W),
\qquad
D(M)_{K_0}=D(M)\otimes_W K_0,
\]
with $F$ $\sigma$-semilinear, $V$ $\sigma^{-1}$-semilinear, and
$FV=VF=p$. The induced increasing weight filtration on $D(M)$ has
cohomological weights $0,1,2$.

The crystalline realization of $1$-motives, its weight filtration,
and its comparison with the de Rham realizations of liftings were
constructed by Andreatta and Barbieri-Viale \cite{ABV}; the
universal-extension construction was further developed by Andreatta
and Bertapelle \cite{AB}. Our filtration is the intrinsic weight
filtration. A Hodge filtration on the characteristic-zero realization
belongs to a choice of deformation, as in the deformation theories
of Messing and of Bertapelle and Mazzari \cite{Messing,BM}.

Our first result is integral full faithfulness:
\begin{equation}\label{eq:intro-integral}
\Hom_k(M,N)\otimes_\Z\Z_p
\xrightarrow{\sim}
\Hom_{\Dieu_W^{\wt}}(D(N),D(M)).
\end{equation}
Here $\Dieu_W^{\wt}$ is the category of weight-filtered Dieudonn\'e
modules defined in Section~\ref{sec:realizations}, with morphisms
commuting with $F,V$ and preserving the filtration. Every
Dieudonn\'e morphism between these realizations preserves the
weight filtration automatically. The filtration therefore records
the geometric structure without restricting the morphism groups.

The rational comparison follows from classical results:
$1$-motives over finite fields split by weights up to isogeny,
their isogeny category is semisimple
\cite[Rem.~3.2.3]{Orgogozo}, and Tate's theorem controls the abelian
part \cite{Tate}. The integral step uses the divisibility criterion
\[
a[n]=0
\quad\Longleftrightarrow\quad
a\in n\Hom_k(M,N),
\qquad
a\in\Hom_k(M,N),\quad n\geq1,
\]
where $a[n]:M[n]\to N[n]$ is the induced map on finite torsion.
This criterion makes the image of realization saturated and upgrades
the rational comparison to \eqref{eq:intro-integral}, without
choosing an integral weight splitting.

The extension data provide a precise measure of the obstruction
to such a splitting. For fixed $L$ and $G$, the kernel of
\[
\Hom_k(L,G)\longrightarrow
\Ext^1_{\BT(k)}
\bigl(L\otimes\Q_p/\Z_p,G[p^\infty]\bigr)
\]
is exactly the prime-to-$p$ subgroup
(Theorem~\ref{thm:extension-kernel}), where $\BT(k)$ denotes the
category of Barsotti--Tate groups over $k$. Thus the
Barsotti--Tate realization detects precisely the $p$-primary part
of the lattice-map extension class. Write $M^\vee=[L^\vee\xrightarrow{u^\vee}G^\vee]$ for the Cartier
dual and put
$r(M):=v_p(\ord(u))$,
with the order of the zero map taken to be $1$. Let $e_i$ be the
canonical rational projector onto the weight-$i$ component of
$D(M)_{K_0}$. We prove that the smallest integer $\delta\geq0$
for which every $p^\delta e_i$ preserves $D(M)$ is
\[
\delta(M)=\max\{r(M),r(M^\vee)\}.
\]
The same power $p^{\delta(M)}$ is the exponent of the finite quotient
\[
\Delta(M)=
D(M)\Big/
\bigoplus_i\bigl(D(M)\cap e_iD(M)_{K_0}\bigr)
\]
(Theorem~\ref{thm:exact-denominators}). This means that $p^{\delta(M)}\Delta(M)=0$, and if $\delta(M)>0$, then $p^{\delta(M)-1}\Delta(M)\neq 0$. Frobenius resultants give
uniform bounds for this exact geometric invariant, with sharper
bounds obtained from minimal polynomials and reduced resultants.
The use of resultants to control integral gluing has a related
precedent in Howe--Lauter \cite[\S2]{HL}. Our elliptic example over
$\F_5$ attains the denominator bound; taking its direct sum with
itself exhibits a strict improvement over the
characteristic-polynomial bound
(Example~\ref{ex:sharp-bound}).

The arithmetic of morphisms becomes explicit through semilinear
descent. Put $\phi_M:=F_M^f$. Applying Galois descent to the
realization gives
\begin{equation}\label{eq:intro-descent}
\Hom_k(M,N)\otimes_\Z K_0
\xrightarrow{\sim}
\left\{
h\in\Hom_{K_0}(D(N)_{K_0},D(M)_{K_0}):
\phi_Mh=h\phi_N
\right\}.
\end{equation}
Thus passing from semilinear Frobenius to linear Frobenius
corresponds to extending the coefficient field from $\Q_p$ to $K_0$.
Over $\F_{q^n}$, the comparison involves $\phi_M^n$ and $\phi_N^n$.
Consequently, the Frobenius eigenvalues and their multiplicities
determine the ranks of morphism groups over every finite extension.
The orders of eigenvalue ratios that are roots of unity determine
the minimal extension over which all geometric morphisms between
$M$ and $N$ are defined
(Theorem~\ref{thm:field-definition}). For supersingular elliptic
curves, an explicit matrix calculation illustrates how a
quaternion division algebra descends from a split Frobenius
centralizer.

We characterize the semisimple isocrystals realized by whole
$1$-motives through cyclotomic conditions in weights $0,2$ and
Honda--Tate multiplicities in weight $1$ (Theorem~\ref{thm:essential-image}). These conditions distinguish
the essential image from the larger category obtained by allowing
direct summands after $p$-adic scalar extension. After scalar
extension and idempotent completion, the equivalence of Orgogozo
and Barbieri-Viale--Kahn \cite{Orgogozo,BVK} yields derived
realization equivalences for motives generated by curves
(Theorem~\ref{thm:derived}). The target category matters: for isocrystals with
semisimple linear Frobenius, we canonically identify ambient
$\Ext^1$ groups with the corresponding $\Hom$ groups
(Proposition~\ref{prop:ambient-ext}). In particular, the unit
isocrystal has nontrivial self-extensions in the ambient category,
whereas extensions within the semisimple realization category split. For smooth projective curves, the morphism comparison identifies
degree-one crystalline morphisms with scalar combinations of
divisor correspondences
(Corollary~\ref{cor:curves}).

\section{Realizations and conventions}\label{sec:realizations}
\subsection{The torsion realization}
A lattice over $k$ is an \emph{\'etale} locally constant group scheme that becomes $\Z^r$ over a finite separable extension. Equivalently, it is a finite-rank free $\Z$-module with continuous Galois action, with the discrete topology on the module. The image of that action is finite. A semi-abelian variety is an extension $0\to T\to G\to A\to0$ of an abelian variety by a torus.

Morphisms $f:M\to N$, for $M=[L\xrightarrow u G]$ and $N=[L'\xrightarrow{u'}G']$, are pairs $(f_L,f_G)$ with $f_Gu=u'f_L$. We write $\Mone(k)$ for this additive category and $\A=\Mone(k)\otimes\Q$ for its isogeny category.

For an integer $n\geq1$, define the fppf quotient
\begin{equation}\label{eq:torsion}
 M[n]=
 \frac{\{(\ell,g)\in L\times G:u(\ell)=ng\}}
      {\{(n\ell,u(\ell)):\ell\in L\}}.
\end{equation}
There is a natural exact sequence
\begin{equation}\label{eq:torsion-exact}
 0\longrightarrow G[n]\longrightarrow M[n]
 \longrightarrow L/nL\longrightarrow0.
\end{equation}
The last term is a quotient of the lattice, whereas the kernel of $[n]$ on the lattice itself is zero. The groups $M[n]$ are finite flat group schemes, and the maps
$[(\ell,g)]\mapsto[(p\ell,g)]$ from $M[p^r]$ to $M[p^{r+1}]$ define the $p$-divisible group (or Barsotti--Tate group) $M[p^\infty]$. Its height is
\[
 \rk L+2\dim A+\dim T.
\]
These constructions are functorial; see \cite[\S1]{ABV}.

\subsection{A weight-filtered Dieudonn\'e category}
Let \(\mathbb{D}(X)\) denote the contravariant crystalline
Dieudonn\'e crystal associated with a \(p\)-divisible group \(X\);
see \cite{BerthelotMessing1979}. We use the contravariant convention
throughout. When \(k\) is perfect, we write
\[
D(X)
:=
\varprojlim_{n}\mathbb{D}(X)\bigl(W_n(k)\bigr),
\]
where \(W_n(k)\to k\) is equipped with its canonical divided-power
structure.

Let $\Dieu_W$ be the category of finite free $W$-modules $D$ equipped with a $\sigma$-semilinear operator $F$, and a $\sigma^{-1}$-semilinear operators $V$ satisfying $FV=VF=p$. Its morphisms commute with Both operators. This is an exact category with the exact sequences inherited from $W$-modules. It is not an abelian category of finite free modules, since multiplication by $p$ has a torsion cokernel. 

Let $\Dieu_W^{\wt}$ be the category of such modules with an increasing filtration
\[
0=\wt_{-1}D\subseteq\wt_0D\subseteq\wt_1D \subseteq\wt_2D=D
\]
by $F,V$-stable saturated $W$-submodules. Its morphisms are $W$-linear maps that commute with both operators and preserve the filtration. We use the exact structure consisting of sequences that are exact on the underlying modules and on the filtered pieces. A saturated submodule has a torsion-free quotient, hence a free quotient here.

The motivic weight filtration is
\[
 W_{-3}M=0,\quad W_{-2}M=[0\to T],\quad
 W_{-1}M=[0\to G],\quad W_0M=M.
\]
The corresponding graded pieces are denoted by
\(\gr_i^W M:=W_iM/W_{i-1}M\).
Contravariance gives the following filtration, with reversed indices:
\begin{equation}\label{eq:weight-definition}
 \wt_iD(M)=D(M/W_{-i-1}M)\subseteq D(M)
 \quad(i=0,1,2).
\end{equation}
Thus
\begin{align*}
 \wt_0D(M)&=D([L\to0]),\\
 \wt_1D(M)&=D([L\to A]),\\
 \gr_i^{\wt}D(M)&\cong
 \begin{cases}D(L),&i=0,\\D(A),&i=1,\\D(T),&i=2.
 \end{cases}
\end{align*}
The injections and the saturatedness in \eqref{eq:weight-definition} follow from the exact sequences of $p$-divisible groups and the Dieudonn\'e anti-equivalence \cite{Demazure}. We have therefore defined a contravariant functor
\[
 D:\Mone(k)^{\op}\longrightarrow\Dieu_W^{\wt}.
\]
No choice of lifting enters this definition.

\subsection{The special-fibre Hodge filtration and deformation data}
For $H=M[p^\infty]$ the canonical Hodge exact sequence is a sequence over $k$:
\begin{equation}\label{eq:hodge-modp}
 0\longrightarrow\omega_H\longrightarrow D(M)/pD(M)
 \longrightarrow\Lie(H^\vee)\longrightarrow0.
\end{equation}
In particular, a Hodge filtration exists in characteristic $p$. A lift of $H$ to an appropriate divided-power thickening carries additional data lifting this filtration. Grothendieck--Messing theory describes the deformation problem in these terms \cite[Ch.~V]{Messing}; the corresponding Serre--Tate equivalence for $1$-motives is \cite[Thm.~1.1]{BM}. The crystalline--de Rham comparison (\cite[Theorem A']{ABV}) for a chosen lifting does not make its Hodge filtration independent of the lifting.

There is a simple algebraic check on this distinction. If $\rho:D\twoheadrightarrow D/pD$ and $U\subseteq D/pD$, then
\[
 pD\subseteq\rho^{-1}(U),\qquad
 \rho^{-1}(U)[1/p]=D[1/p].
\]
Consequently, taking the inverse image of $\omega_H$ in \eqref{eq:hodge-modp} cannot produce the expected proper Hodge subspace over $K_0$. For an \'etale $H$, this inverse image is $pD$, even though $\omega_H=0$.

Throughout the paper, the filtration over $W$ and $K_0$ is \eqref{eq:weight-definition}.
\subsection{Cartier Duality}
Cartier duality for $1$-motives is defined by the Poincar\'e biextension \cite[\S10.2]{Deligne}. Its perfect finite-level pairings identify
\[
 M^\vee[p^r]\cong M[p^r]^\vee.
\]
Passing to the associated Barsotti--Tate groups and then to their
Dieudonn\'e modules, these pairings yield a perfect duality compatible
with the weight filtrations \cite[\S1.3]{ABV}.

Put $\mathbb L:=D(\Gm)=(W,p\sigma,\sigma^{-1})$. The induced crystalline pairing is perfect and yields
\[
 D(M^\vee)\cong\Hom_W(D(M),\mathbb L),\qquad
 \wt_iD(M^\vee)=(\wt_{1-i}D(M))^\perp.
\]
Here the Frobenius on the Hom module is $h\mapsto p\sigma\circ h\circ F_M^{-1}$ after inverting $p$, and the integral structure is the one supplied by Cartier duality. In particular, duality exchanges weights $0$ and $2$ and preserves weight $1$.

\section{Integral full faithfulness}\label{sec:fullfaithfulness}
\subsection{The rational structure}
Recall that \(\A=\Mone(k)\otimes\Q\), and let \(\gr^W\) denote the
associated-graded functor for the motivic weight filtration. Thus, if
\(M=[L\to G]\), with toric part \(T\) and abelian quotient \(A\), then
\(\gr^W(M)=(L,A,T)\).
\begin{proposition}[Finite-field splitting; cf. Orgogozo]\label{prop:splitting}
The associated-graded functor is an equivalence
\[
 \A\simeq
 (\text{$k$-lattices up to isogeny})\times
 (\text{abelian varieties over $k$ up to isogeny})\times
 (\text{$k$-tori up to isogeny}).
\]
Every $M=[L\to G]$ has a unique splitting by weights inducing the identity on its graded pieces. In particular, $\A$ is semisimple abelian.
\end{proposition}
\begin{proof}
Choose a finite extension $k'/k$ that splits $L$ and $T$. The morphism $u$ is determined over $k'$ by finitely many elements of the finite group $G(k')$, so some integer $n>0$ kills $u$. The maps
\[
 [L\xrightarrow0G]\xrightarrow{(n,1)}[L\xrightarrow uG],\qquad
 [L\xrightarrow uG]\xrightarrow{(1,n)}[L\xrightarrow0G]
\]
have both composites equal to $n$. This splits off the lattice in $\A$.

The class of $G_{k'}$ lies in
\[
 \Ext^1_{k'}(A_{k'},\Gm^{\dim T})
   \cong A^\vee(k')^{\dim T},
\]
by the Weil--Barsotti formula. Hence it is torsion. After multiplying by an integer, a splitting of the corresponding pullback extension gives $s:A_{k'}\to G_{k'}$ with $\pi s=m\id_A$. Such a map is unique: the difference of two choices lies in $\Hom(A_{k'},T_{k'})=0$. Its Galois conjugates are therefore equal, so $s$ descends to $k$. Then $s/m$ splits the semi-abelian extension up to isogeny.

Morphisms between distinct pure weights vanish. For lattice versus group pieces this follows from the definition of morphisms of $1$-motives. For the torus and abelian pieces one uses $\Hom(A,T)=\Hom(T,A)=0$. The first equality follows from properness. For the second, split the torus and restrict to each one-parameter subgroup: a morphism $\Gm\to A$ extends to $\mathbb P^1$ by properness, and every morphism $\mathbb P^1\to A$ is constant. These vanishings prove uniqueness and functoriality of the splitting, and establish the product equivalence. Rational lattice representations factor through a finite group and are semisimple by Maschke's theorem; the torus case follows by character duality, and the abelian case by Poincar\'e complete reducibility. This recovers \cite[Rem.~3.2.3]{Orgogozo}.
\end{proof}

Let $\Isoc(k)$ be the $\Q_p$-linear abelian category of finite-dimensional $K_0$-vector spaces with invertible $\sigma$-semilinear $F$. A morphism \(h\) between two isocrystals in \(\Isoc(k)\) is a
\(K_0\)-linear intertwiner of Frobenius operators, meaning that
$hF_D=F_{D'}h$. On an isocrystal, $V=pF^{-1}$ is redundant. For a realization, write $\phi_M=F_M^f$ and
\begin{equation}\label{Frobenius-integral-polynomial}
 P_{M,i}(X)=\det(X-\phi_M\mid\gr_i^{\wt}D(M)_{K_0}),
 \qquad P_M=\prod_{i=0}^2P_{M,i}.
\end{equation}
These are monic integral polynomials. Their roots have, at every complex embedding, absolute values
\begin{equation}\label{eq:weights}
 1,\quad q^{1/2},\quad q
 \quad\text{for }i=0,1,2,
\end{equation}
respectively. The roots in weight $0$ are roots of unity, and those in weight $2$ are $q$ times roots of unity. For the abelian graded piece, the assertion concerning the absolute values of the Frobenius eigenvalues is the Weil's theorem; the crystalline identification reviewed in \cite[Chs.~1--2]{Waterhouse}.

The operator \(\phi_M\) is semisimple
on each pure graded piece. This is immediate for lattices, whose
Frobenius action has finite order, and for tori by character duality.
For an abelian variety \(A\), Frobenius lies in the centre of the
semisimple algebra isogeny endomorphism algebra, so its
minimal polynomial is separable. Finally,
Proposition~\ref{prop:splitting} decomposes \(M\) into its pure
weight pieces up to isogeny. Hence \(\phi_M\) is semisimple on
\(D(M)_{K_0}\).

\begin{lemma}\label{lem:automatic-weights}
Every morphism in $\Isoc(k)$ between crystalline realizations preserves their weight filtrations. The same is true integrally in $\Dieu_W$.
\end{lemma}
\begin{proof}
Let \(h\colon D(N)_{K_0}\to D(M)_{K_0}\) be a morphism in
\(\Isoc(k)\). Then \(h\phi_N=\phi_Mh\), so after extending scalars to
a splitting field, \(h\) maps each eigenspace of \(\phi_N\) into the
eigenspace of \(\phi_M\) with the same eigenvalue. Semisimplicity and
\eqref{eq:weights} therefore imply
$h\bigl(\wt_iD(N)_{K_0}\bigr)\subseteq\wt_iD(M)_{K_0}$.

For an integral morphism, the equality
\[
\wt_iD(M)=D(M)\cap\wt_iD(M)_{K_0},
\]
which follows from saturatedness, gives the result over \(W\).
\end{proof}

\begin{proposition}[Rational comparison]\label{prop:rational}
For $M,N\in\Mone(k)$, the natural map
\[
 \Hom_k(M,N)\otimes_\Z\Q_p
 \xrightarrow{\sim}\Hom_{\Isoc(k)}(D(N)_{K_0},D(M)_{K_0})
\]
is an isomorphism.
\end{proposition}
\begin{proof}
Proposition~\ref{prop:splitting} reduces the source to pure weights, and \eqref{eq:weights} does the same for the target. For lattices, finite \'etale group schemes identify the comparison with the Galois-equivariant comparison for finite-rank lattice representations. Hence the equivariant homomorphisms are the kernel of finitely many integral linear maps, and flatness of \(\Z\to\Q_p\) shows that this map is an isomorphism. Character duality for tori, together with Cartier duality, reduces the toric case to the lattice case \cite[Exp.~X, Cor.~7.2]{SGA3}. For abelian varieties, the result is the $p$-adic form of Tate's homomorphism theorem \cite{Tate}; its precise Dieudonn\'e formulation is recorded in \cite[\S1.2]{Waterhouse}. These comparisons respect composition and the canonical splitting.
\end{proof}

\subsection{Divisibility at a finite level}
The following lemma works over any field; its proof uses the fppf quotient in \eqref{eq:torsion}.
\begin{lemma}[Finite-level divisibility]\label{lem:divisibility}
For $1$-motives $M,N$ over a field and $n\geq1$, a morphism $a:M\to N$ satisfies
\[
 a[n]=0\quad\Longleftrightarrow\quad a=nb
 \text{ for some }b\in\Hom(M,N).
\]
The morphism $b$, when it exists, is unique.
\end{lemma}
\begin{proof}
One implication follows because $N[n]$ is killed by $n$. Conversely, write $a=(a_L,a_G)$ and suppose $a[n]=0$. The quotient map in \eqref{eq:torsion-exact} shows that $a_L$ is zero modulo $n$, hence $a_L=nb_L$ for a unique lattice map $b_L$. Restriction to $G[n]$ gives $a_G|_{G[n]}=0$. Since $[n]:G\to G$ is an fppf epimorphism with kernel $G[n]$, $a_G$ factors uniquely as $a_G=b_G\circ[n]=nb_G$.

It remains to check compatibility with $u$. Let $\ell$ be a local section of $L$, and choose $g$ fppf locally with $ng=u(\ell)$. The image of $[(\ell,g)]$ under $a[n]$ is zero. By \eqref{eq:torsion}, this means
\[
 (nb_L\ell,a_Gg)=(n\ell',u'(\ell'))
\]
locally for some $\ell'\in L'$. Since $L'$ is torsion-free, $\ell'=b_L\ell$. Thus
\[
 b_Gu(\ell)=b_G(ng)=a_Gg=u'b_L(\ell).
\]
This equality descends, so $(b_L,b_G)$ is a morphism of $1$-motives. Uniqueness follows from torsion-freeness of the lattice Hom group and from the epimorphism $[n]$ on $G$.
\end{proof}

\begin{corollary}\label{cor:finite-injection}
For every $n\geq1$, the natural map
\[
 \Hom_k(M,N)/n\Hom_k(M,N)
 \longrightarrow\Hom_k(M[n],N[n])
\]
is injective. The target here is the group of finite-group-scheme morphisms; surjectivity at a fixed level is not asserted.
\end{corollary}

\begin{theorem}[Integral full faithfulness]\label{thm:integral}
For $M,N\in\Mone(k)$ there are natural isomorphisms
\begin{align*}
 \Hom_k(M,N)\otimes\Z_p
 &\xrightarrow{\sim}\Hom_{\BT(k)}(M[p^\infty],N[p^\infty])\\
 &\xrightarrow{\sim}\Hom_{\Dieu_W^{\wt}}(D(N),D(M)).
\end{align*}
For endomorphisms the second identification reverses composition.
\end{theorem}
\begin{proof}
Put $H=\Hom_k(M,N)$. This is a finite-rank free abelian group. Indeed, $\Hom(G,G')$ injects into
$\Hom(T,T')\oplus\Hom(A,A')$: a map inducing zero on both constituents factors through a map $A\to T'$, which is zero. The latter groups and the lattice Hom groups are finite free, and $H$ is a subgroup of their direct sum.

Put $B=\Hom_{\Dieu_W}(D(N),D(M))$. Note that on $B$ the weight filtrations are temporarily forgotten. This is a finite free $\Z_p$-module, being a $\Z_p$-submodule of the finite free module $\Hom_W(D(N),D(M))$. Clearing powers of $p$ in a rational Dieudonn\'e morphism shows that
\[
 B\otimes\Q_p=\Hom_{\Isoc(k)}(D(N)_{K_0},D(M)_{K_0}).
\]
Proposition~\ref{prop:rational} therefore makes $H\otimes\Z_p\to B$ injective with finite $p$-primary cokernel.

This inclusion is saturated. To see this, take a class in $(H\otimes\Z_p)/p$ whose image in $B/pB$ vanishes and represent it by $a\in H$. Then $D(a)\in pB$. Reduction modulo $p$ and the finite Dieudonn\'e anti-equivalence imply $a[p]=0$. Lemma~\ref{lem:divisibility} gives $a\in pH$. Thus
\[
 (H\otimes\Z_p)/p\longrightarrow B/pB
\]
is injective. A finite $p$-primary cokernel of an inclusion of free $\Z_p$-modules can have this property only if it is zero. Indeed, if cokernel does not vanish, it would contain an element of order \(p\).
Lifting this element to \(b\in B\), we would have
\(pb\in H\otimes\Z_p\setminus p(H\otimes\Z_p)\), while the image of \(pb\) in \(B/pB\) would be zero. This would contradict the injectivity. It follows that $H\otimes\Z_p\xrightarrow{\sim}B$. By the contravariant Dieudonn\'e anti-equivalence,
\[
B\cong
\Hom_{\BT(k)}(M[p^\infty],N[p^\infty]),
\]
which proves the first isomorphism. Finally, Lemma~\ref{lem:automatic-weights} shows that every element of \(B\)
preserves the weight filtration, so $B=\Hom_{\Dieu_W^{\wt}}(D(N),D(M))$.
\end{proof}

\begin{remark}\label{rem:coefficients}
The word ``full'' in Theorem~\ref{thm:integral} refers to $p$-adically completed morphism groups. A crystalline morphism is a unique element of $\Hom_k(M,N)\otimes\Z_p$; it need not be one integral geometric morphism. For example, multiplication by an element of $\Z_p\setminus\Z$ on the constant rank-one realization is already such a morphism.
\end{remark}

\section{Extension classes and integral weight splitting}\label{sec:integral-structure}
\subsection{The exact kernel on lattice extensions}
For an abelian torsion group $C$, write $C[p^\infty]$ for its $p$-primary subgroup and $C[p']$ for its prime-to-$p$ subgroup. In the next theorem, $\Ext^1_{\BT(k)}$ means Yoneda $1$-extensions of $p$-divisible groups.

For \(u\in \Hom_k(L,G)\), let \(M_u=[L\xrightarrow{u}G]\). Taking
\(p\)-power torsion gives an exact sequence
\[
0\longrightarrow G[p^\infty]
\longrightarrow M_u[p^\infty]
\longrightarrow L\otimes\Q_p/\Z_p
\longrightarrow0.
\]
Let \(\kappa(u)\) denote its extension class. In this way, one obtains
a homomorphism
\[
\kappa\colon \Hom_k(L,G)\longrightarrow
\Ext^1_{\BT(k)}
\bigl(L\otimes\Q_p/\Z_p,G[p^\infty]\bigr).
\]
\begin{theorem}\label{thm:extension-kernel}
Let $L$ be a lattice and $G$ a semi-abelian variety over $k$. The finite group $C=\Hom_k(L,G)$ maps by the torsion construction to
 $\Ext^1_{\BT(k)}(L\otimes\Q_p/\Z_p,G[p^\infty])$. Its kernel is $C[p']$. Consequently, $\kappa$ is injective on $C[p^\infty]$. If $u\in C$ has order $p^rm$ with $(m,p)=1$, then $\kappa(u)$ has order exactly $p^r$.
\end{theorem}
\begin{proof}
After a finite extension $k'/k$ splitting $L$, the group $C$ injects into $G(k')^{\rk L}$, so it is finite. The torsion construction respects addition of maps and Baer sum, hence defines the indicated homomorphism.

Suppose $mu=0$ with $(m,p)=1$. Choose integers $c_n$ with $p^nc_n\equiv1\pmod m$, and put $v_n=c_nu$. Thus $p^nv_n=u$ and $pv_{n+1}=v_n$. Define
\[
 s_n:L/p^nL\longrightarrow [L\xrightarrow uG][p^n],
 \qquad \bar\ell\longmapsto[(\ell,v_n(\ell))].
\]
Changing $\ell$ by $p^nx$ changes the pair by $(p^nx,u(x))$, so $s_n$ is well-defined. The equality $pv_{n+1}=v_n$ shows that these sections commute with the transition maps. They split the extension of $p$-divisible groups.

Conversely, suppose that extension splits. Base change to $k'$ and take a basis $e_1,\dots,e_r$ of $L(k')$. The fibre of $M[p^n]\to L/p^nL$ above $\bar e_i$, with the representative $e_i$ fixed, is canonically the scheme of solutions to $p^ng=u(e_i)$. A $k'$-rational section therefore gives $g_{i,n}\in G(k')$ with $p^ng_{i,n}=u(e_i)$. In a finite abelian group $B$,
\[
 \bigcap_{n\geq1}p^nB=B[p'].
\]
Thus every $u(e_i)$ has prime-to-$p$ order. Their common exponent kills $u$ after base change and therefore over $k$. This proves the kernel assertion. Decomposing $C=C[p^\infty]\oplus C[p']$ gives the final statement.
\end{proof}

\begin{corollary}\label{cor:extension-splitting}
Let $M=[L\xrightarrow{u}G]$. The sequence
\[
 0\longrightarrow G[p^\infty]\longrightarrow M[p^\infty]
 \longrightarrow L\otimes\Q_p/\Z_p\longrightarrow0
\]
splits if and only if $u$ has prime-to-$p$ order. If it splits, its splitting is unique. For a semi-abelian extension $0\to T\to G\to A\to0$, its $p$-divisible extension splits if and only if its extension class has prime-to-$p$ order.
\end{corollary}
\begin{proof}
The first assertion is Theorem~\ref{thm:extension-kernel}. Two splittings differ by a morphism from $L\otimes\Q_p/\Z_p$ to $G[p^\infty]$. By Theorem~\ref{thm:integral}, this group is
$\Hom_{\Mone(k)}([L\to0],[0\to G])\otimes\Z_p=0$. Hence the splitting is unique. Cartier duality identifies semi-abelian extensions of $A$ by $T$ with maps $X^*(T)\to A^\vee$. It carries the associated $p$-divisible extensions into one another. The final assertion follows from the first one applied to this dual $1$-motive.
\end{proof}

\subsection{Resultant bounds}
By Proposition~\ref{prop:splitting}, the weight filtration on the
rational realization admits a canonical decomposition
\begin{equation}\label{eq:weight decomposition_K0}
D(M)_{K_0}
=
\bigoplus_{i=0}^2 \gr_i^{\wt}D(M)_{K_0}.
\end{equation}
Let \(e_i\) denote the projector onto the weight-\(i\) summand. These
projectors commute with Frobenius and Verschiebung, but they need not
preserve the integral lattice \(D(M)\). The next theorem gives an
explicit bound for the powers of \(p\) required to make the
projectors integral.

We normalize $v_p(p)=1$ and disregard weights of rank zero. For monic polynomials \(P,Q\in\Z[X]\), write
\(\Res(P,Q)\) for their resultant, namely the determinant of their
Sylvester matrix, \textit{ie.}, if \(P\) and \(Q\) are with roots \(\alpha_i\) and
\(\beta_j\), respectively, then
\[
\Res(P,Q)=\prod_{i,j}(\alpha_i-\beta_j).
\]
The argument is an integral Chinese-remainder construction. Related uses of resultants occur in Howe--Lauter's bound for gluing exponents \cite[Lem.~2.10 and Prop.~2.8]{HL} and in the unit-resultant criterion for integral orthogonal idempotents in \cite[Prop.~6]{MarsegliaScholl}. Here the factors are the pure weight summands of \(D(M)_{K_0}\), and B\'ezout identities for their Frobenius polynomials bound the denominators of the rational weight projectors and the exponent of the integral splitting defect.

Recall the integral Frobenius polynomials from \eqref{Frobenius-integral-polynomial}.
\begin{theorem}[Integral bound for the weight projectors]\label{thm:resultants}
Let $D=D(M)$,\, $I:=\bigl\{i\in\{0,1,2\}:\gr_i^{\wt}D\neq0\bigr\}$, and put
\[
 P_i=P_{M,i},\qquad Q_i=\prod_{j\in I\setminus\{i\}}P_j,
 \qquad R_i=\Res(P_i,Q_i)\in\Z\setminus\{0\}.
\]
Let $e_i$ denote the canonical projector onto the weight-$i$ summand of $D_{K_0}$. Then
\[
 p^d e_i(D)\subseteq D,
 \qquad d=\max_{i\in I}v_p(R_i).
\]
For $D=0$ we set $d=0$.
The modules $D_i=D\cap e_iD_{K_0}$ are saturated Dieudonn\'e submodules, and
\begin{equation}\label{eq:integral-index}
 \bigoplus_{i\in I}D_i\ \lhook\joinrel\longrightarrow D
\end{equation}
has cokernel killed by $p^d$, of $W$-length at most $d\rk_WD$. If all pairwise resultants $\Res(P_i,P_j)$ are $p$-adic units, the weight filtration splits canonically over $W$.
\end{theorem}
\begin{proof}
The polynomials $P_i$ are pairwise relatively prime in $\Q[X]$ by their distinct complex weights. The integral resultant identity gives $a_i,b_i\in\Z[X]$ such that
\[
 a_iP_i+b_iQ_i=R_i.
\]
Indeed, on the weight-\(i\) summand, Cayley--Hamilton gives
\(P_i(\phi_M)=0\), so the resultant identity yields
$b_i(\phi_M)Q_i(\phi_M)=R_i\id$.
On every weight-\(j\) summand with \(j\neq i\), the polynomial \(P_j\)
divides \(Q_i\), and hence \(Q_i(\phi_M)=0\). Therefore
\[
e_i=R_i^{-1}b_i(\phi_M)Q_i(\phi_M)
\]
is the projector onto the weight-\(i\) summand.
The numerator preserves $D$ because $\phi_M=F_M^f$ does. Since $p^d/R_i\in\Z_p$, the asserted denominator bound follows. All these projectors commute with $F$ and $V$, and their images are the distinct rational weight summands. The intersections $D_i$ are therefore stable and saturated.

For $x\in D$, each $p^de_ix$ belongs to $D_i$ and
\[
 p^dx=\sum_{i\in I}p^de_ix.
\]
This proves the annihilator assertion for the cokernel of \eqref{eq:integral-index}. That cokernel is a quotient of $D/p^dD$, which gives the length bound. Finally, $R_i$ is the product, up to sign, of the pairwise resultants involving $P_i$. By the multiplicativity of the resultant,
\[
R_i=\Res(P_i,Q_i)
   =\prod_{j\in I\setminus\{i\}}\Res(P_i,P_j).
\]
Hence, if every pairwise resultant is prime to \(p\), then each
\(R_i\) is a unit in \(\Z_p\), so \(d=0\). It follows that
\(e_i(D)\subseteq D\) for every \(i\). Since the \(e_i\) are
orthogonal idempotents satisfying
$\sum_{i\in I}e_i=1$,
we obtain
$D=\bigoplus_{i\in I}e_iD$.
Moreover, \(e_iD=D\cap e_iD_{K_0}=D_i\), and therefore
\[
\wt_rD=\bigoplus_{\substack{i\in I\\ i\leq r}}D_i.
\]
The projectors \(e_i\) are determined by the canonical rational
weight decomposition, so the resulting integral splitting is
canonical.
\end{proof}

\begin{remark}
This bound measures a possible integral obstruction; a resultant divisible by $p$ does not force nonsplitting. The rational splitting is canonical for every $M$, whereas integral splitting depends on the extension data. By Theorem~\ref{thm:integral}, the operators $p^de_i$ correspond to elements of $\End_k(M)\otimes\Z_p$. This statement does not assert that the projectors $e_i$ themselves are integral.
\end{remark}

\begin{example}[A rational point on an abelian variety]\label{ex:point-order}
Let $M_P=[\Z\xrightarrow{1\mapsto P}A]$ with $P\in A(k)$. Put $R=\End_k(A)$ and let $P_p$ be the $p$-primary component of $P$. Directly from the definition of a morphism,
\begin{equation}\label{eq:point-end}
 \End_k(M_P)=\{(n,a)\in\Z\oplus R:a(P)=nP\}.
\end{equation}
Flatness of $\Z_p$ gives the exact integral description
\begin{equation}\label{eq:point-end-p}
 \End_k(M_P)\otimes\Z_p
 =\{(z,b)\in\Z_p\oplus(R\otimes\Z_p):b(P_p)=zP_p\}.
\end{equation}
The actions on $P_p$ factor through a finite quotient, so the right-hand side is well-defined. If $P_p$ has order $p^r$, the Barsotti--Tate extension has order $p^r$ by Theorem~\ref{thm:extension-kernel}. Here the only relevant resultant is
\[
 \Res(X-1,P_A(X))=P_A(1)=\#A(k),
\]
so Theorem~\ref{thm:resultants} gives the uniform bound  $p^{v_p(\#A(k))}e_i(D(M_P))\subseteq D(M_P)$. This bound depends only on $A$, not on the point $P$. Formula \eqref{eq:point-end-p} gives the sharper point-dependent exponent:
\(p^se_0\) is integral if and only if \(p^sP_p=0\), so the exact
exponent is \(v_p(\ord(P))\).
\end{example}

\subsection{Exact denominators from the geometric extension data}
Write $M^\vee=[L^\vee\xrightarrow{u^\vee}G^\vee]$ for the Cartier-dual $1$-motive. The notation $L^\vee$ here denotes its lattice, namely $X^*(T)$, rather than the ordinary dual of $L$. Both $u$ and $u^\vee$ have finite order. Define
\[
 r(M)=v_p(\ord(u)),\qquad
 \delta(M)=\max\{r(M),r(M^\vee)\},
\]
where the zero map has order $1$. Put $e_i=0$ if weight $i$ is absent. The exponent of the zero module is taken to be $1$.
\begin{theorem}[Exact integral splitting defect]\label{thm:exact-denominators}
Let $D=D(M)$ and $D_i=D\cap e_iD_{K_0}$. For every integer $s\geq0$,
\begin{align*}
 p^se_0(D)\subseteq D&\quad\Longleftrightarrow\quad s\geq r(M),\\
 p^se_2(D)\subseteq D&\quad\Longleftrightarrow\quad s\geq r(M^\vee).
\end{align*}
Consequently, $\delta(M)$ is the exact common denominator exponent of the three weight projectors. The finite module
\[
 \Delta(M)=D/(D_0\oplus D_1\oplus D_2)
\]
has exponent exactly $p^{\delta(M)}$ and length at most $\delta(M)\rk_WD$. Finally, the weight filtration splits over $W$ if and only if both $u$ and $u^\vee$ have prime-to-$p$ order.
\end{theorem}
\begin{proof}
Let $C=\Hom_k(L,G)$. The definition of a $1$-motive morphism gives an exact sequence at the first two terms
\[
 0\longrightarrow\End_k(M)\longrightarrow
 \End_k(L)\oplus\End_k(G)
 \xrightarrow{\ \partial\ } C,\qquad
 \partial(a,b)=bu-ua.
\]
The canonical rational splitting of the lattice from $G$ identifies the weight-$0$ projector on $M$ with $(\id_L,0)$. Flatness of $\Z_p$, tensoring the
preceding exact sequence with \(\Z_p\), and the identity $C\otimes\Z_p=C[p^\infty]$ show that
\[
 p^s(\id_L,0)\in\End_k(M)\otimes\Z_p
 \quad\Longleftrightarrow\quad p^su_p=0,
\]
where $u_p$ is the $p$-primary component of $u$. Theorem~\ref{thm:integral} turns this condition into \[
p^s\varepsilon_0\in
\operatorname{End}_k(M)\otimes\mathbf Z_p
\quad\Longleftrightarrow\quad
p^se_0\in\operatorname{End}_{\Dieu_W}(D),
\]
where under rational crystalline realization, \(\varepsilon_0\) corresponds to the weight-\(0\) projector \(e_0\). This proves the first equivalence without choosing an integral splitting.

The perfect integral Cartier pairing identifies the adjoint of the weight-$2$ projector on $D(M)$ with the weight-$0$ projector on $D(M^\vee)$. Hence, an operator preserves \(D(M)\) if and only if its adjoint
preserves \(D(M^\vee)\). Apply the first equivalence to $M^\vee$ to obtain the second.

Since $e_1=1-e_0-e_2$, all three $p^se_i$ preserve $D$ if $s\geq\delta(M)$. Conversely, if all three operators preserve
\(D\), then the equivalences for \(e_0\) and \(e_2\) imply
$s\geq\max\{r(M),r(M^\vee)\}=\delta(M)$. Hence \(\delta(M)\) is the exact common denominator exponent. Moreover,
\[
 p^s\Delta(M)=0
 \quad\Longleftrightarrow\quad
 p^sD\subseteq\bigoplus_iD_i
 \quad\Longleftrightarrow\quad
 p^se_iD\subseteq D\ \text{for every }i.
\]
For the last implication in either direction, use the unique decomposition in $\bigoplus_i e_iD_{K_0}$. This proves the exact exponent. Put
$S=D_0\oplus D_1\oplus D_2$.
Since
$S[1/p]=D_{K_0}$, the module \(S\) is a full-rank $W$-sublattice of $D$. Thus \(p^nD\subseteq S\) for some \(n\geq0\). Since \(W\) is a
discrete valuation ring with uniformizer \(p\),
$\operatorname{length}_W(D/p^nD)
   =n\operatorname{rk}_W D$.
As \(D/S\) is a quotient of \(D/p^nD\), it has finite \(W\)-length. Since \(p^{\delta(M)}D\subseteq S\), the quotient map induces $D/p^{\delta(M)}D\to\Delta(M)$. Therefore, $\operatorname{length}_W\Delta(M)
\leq
\operatorname{length}_W
   \bigl(D/p^{\delta(M)}D\bigr)
=
\delta(M)\operatorname{rk}_W D$.

Finally, the canonical rational weight decomposition and saturatedness give
$\mathsf W_jD
=
D\cap\bigoplus_{i\leq j}e_iD_{K_0}$.
Hence the weight filtration splits over \(W\) if and only if
\(D=\bigoplus_iD_i\), equivalently if and only if \(\Delta(M)=0\).
Therefore,
\[
\begin{aligned}
\text{the weight filtration splits over \(W\)}
\Longleftrightarrow \Delta(M)=0
\Longleftrightarrow \delta(M)=0\Longleftrightarrow r(M)=r(M^\vee)=0\\
\Longleftrightarrow
\operatorname{ord}(u)\ \text{and}\ 
\operatorname{ord}(u^\vee)\ \text{are prime to }p.
\end{aligned}
\]
\end{proof}

For $M_P$ in Example~\ref{ex:point-order}, the torus part vanishes, so the dual lattice is zero. Put $r=v_p(\ord(P))$. The exact denominator is $p^r$, and in fact
\[
 \Delta(M_P)\cong W/p^rW.
\]
To see the latter statement, one can use $e_1=1-e_0$ to see that projection to weight $0$ identifies $\Delta(M_P)$ with $e_0D/D_0$. Since \(D_0\subseteq e_0D\) are rank-one \(W\)-lattices in
\(e_0D_{K_0}\), their quotient is isomorphic to \(W/p^aW\) for some
\(a\geq0\). Its exponent is \(p^r\) by
Theorem~\ref{thm:exact-denominators}, so \(a=r\). This determines the
underlying \(W\)-module structure of the splitting defect and sharpens
the resultant estimate obtained from
\(\operatorname{Res}(X-1,P_A(X))=\#A(k)\).

\subsection{Reduced resultants and the best bound from Frobenius polynomials}
There is a useful purely algebraic sharpening of Theorem~\ref{thm:resultants}. Let $I=\bigl\{i:\gr_i^{\wt}D(M)\neq0\bigr\}$. For $i\in I$, let $\mu_i\in\Z[X]$ be the minimal polynomial of $\phi_M$ on its weight-$i$ summand and put $U_i=\prod_{j\neq i}\mu_j$. These minimal polynomials are monic, integral, and pairwise relatively prime.
Let \(\rho_i\) be the reduced resultant of \(\mu_i\) and \(U_i\), namely
the unique positive integer satisfying
\[
(\mu_i,U_i)\Z[X]\cap\Z=\rho_i\Z;
\]
see \cite[p.~179]{Pohst} and
\cite[Def.~2.6 and Lem.~2.7]{HL}. The definition makes sense because the usual resultant is a nonzero element of this intersection. For monic polynomials, this reduced resultant is the
``congruence number'' of \cite[Def.~2.6]{TaixesWiese};
loc.~cit., \S~2.3 describes its computation from the
Sylvester matrix.

The role of \(\rho_i\) can be seen from the regular lattice
$B= \Z_p[X]\Big/\Bigl(\prod_{j\in I}\mu_j\Bigr)$,
where \(B\) is regarded as a module over itself and \(\phi\) acts by multiplication by \(X\). An empty product is $1$. After inverting \(p\), the Chinese remainder decomposition defines the factor projector \(e_i\), characterized by
\[
e_i\equiv 1\pmod{\mu_i},
\qquad
e_i\equiv 0\pmod{U_i}.
\]
This projector need not preserve \(B\). The following proposition shows that its exact denominator is \(p^{v_p(\rho_i)}\). Thus \(\rho_i\) gives the optimal bound obtainable solely from the polynomials \(\mu_i\) and the integrality of \(\phi\).

\begin{proposition}[Reduced-resultant bound]\label{prop:reduced-resultants}
With $d_{\mathrm{red}}=\max_{i\in I}v_p(\rho_i)$, or $0$ when $I$ is empty, one has
\[
 \delta(M)\leq d_{\mathrm{red}}\leq d.
\]
Each $\rho_i e_i$ preserves $D$. For the regular lattice
$B=\Z_p[X]/(\prod_i\mu_i)$ with $\phi$ acting by $X$, the least $s$ such that $p^se_iB\subseteq B$ is exactly $v_p(\rho_i)$.
\end{proposition}
\begin{proof}
Choose \(a_i,b_i\in\Z[X]\) such that
$a_i\mu_i+b_iU_i=\rho_i$.
On the weight-\(i\) summand, \(\mu_i(\phi_M)=0\), while on every
weight-\(j\) summand with \(j\neq i\), the polynomial \(\mu_j\) divides
\(U_i\). Hence
$\rho_i e_i=b_i(\phi_M)U_i(\phi_M)$.
The right-hand side preserves \(D\). Since
\(\rho_i/p^{v_p(\rho_i)}\) is a unit in \(W\), it follows that
$p^{v_p(\rho_i)}e_i(D)\subseteq D$.
Thus \(p^{d_{\mathrm{red}}}e_i(D)\subseteq D\) for every \(i\), and
Theorem~\ref{thm:exact-denominators} gives
\(\delta(M)\leq d_{\mathrm{red}}\).

The resultant B\'ezout identity gives $R_i\in(P_i,Q_i)\Z[X]$.
Since \(\mu_i\mid P_i\) and \(U_i\mid Q_i\), we have
\[
(P_i,Q_i)\Z[X]\subseteq(\mu_i,U_i)\Z[X].
\]
Therefore
$R_i\in(\mu_i,U_i)\Z[X]\cap\Z=\rho_i\Z$,
so \(\rho_i\mid R_i\), proving \(d_{\mathrm{red}}\leq d\).

For the final assertion, the Chinese remainder decomposition of
\(B[1/p]\) identifies \(e_i\) with the element that is \(1\) modulo
\(\mu_i\) and \(0\) modulo \(U_i\). Since \(B\) is generated by \(1\)
as a \(\Z_p[X]\)-module,
\[
p^se_iB\subseteq B
\quad\Longleftrightarrow\quad
p^se_i\in B.
\]
This is equivalent to the existence of \(h\in\Z_p[X]\) satisfying
$h\equiv p^s\pmod{\mu_i},
\, \text{ and }
h\equiv0\pmod{U_i},$ or equivalently,
$p^s\in(\mu_i,U_i)\Z_p[X]$.
Since the kernel of
$\Z\longrightarrow\Z[X]/(\mu_i,U_i)$
is \(\rho_i\Z\), flatness of \(\Z_p\) gives
$(\mu_i,U_i)\Z_p[X]\cap\Z_p
=\rho_i\Z_p
=p^{v_p(\rho_i)}\Z_p$.
Consequently,
\[
p^se_iB\subseteq B
\quad\Longleftrightarrow\quad
s\geq v_p(\rho_i),
\]
The exact denominator follows.
\end{proof}
The regular lattice in here need not admit the additional $F,V$ operators of a Dieudonn\'e lattice. Its role is to identify the information contained in the polynomial argument. For geometric realizations, Theorem~\ref{thm:exact-denominators} gives the exact answer, and Example~\ref{ex:sharp-bound} below attains the polynomial bound.

\section{Frobenius descent and fields of definition}\label{sec:descent}
\subsection{Descent from a linear Frobenius}
We first prove the linear algebra needed to pass from $F$ to $F^f$. The following is the cyclic case of Galois descent for vector spaces.
Its spanning assertion is Speiser's lemma; see
\cite[Lem.~2.3.8]{GilleSzamuely}. We include a direct proof here for completeness.

\begin{lemma}[Semilinear descent]\label{lem:descent}
Let $C$ be a finite-dimensional $K_0$-vector space with a $\sigma$-semilinear map $\Theta$ satisfying $\Theta^f=1$. Then
\[
 C^{\Theta=1}\otimes_{\Q_p}K_0\xrightarrow{\sim}C,
 \qquad c\otimes a\longmapsto ac.
\]
\end{lemma}
\begin{proof}
Since \(K_0/\Q_p\) is cyclic Galois with group generated by \(\sigma\), the powers of \(\Theta\) define a semilinear Galois action on \(C\).
Vectors fixed by $\Theta$ that are linearly independent over $\Q_p$ remain independent over $K_0$. Otherwise choose a relation of minimal length, normalize one coefficient to $1$, apply $\Theta$, and subtract. Minimality forces all coefficients to be $\sigma$-fixed, contradicting independence over $\Q_p$.

It remains to prove that the fixed vectors span \(C\) over \(K_0\).  Indeed, for $a\in K_0$ and $v\in C$, the vector $\sum_{i=0}^{f-1}\Theta^i(av)$ is fixed by $\Theta$. If a $K_0$-linear map $\lambda$ vanishes on the span of the fixed vectors, then
\[
 \sum_{i=0}^{f-1}\sigma^i(a)\lambda(\Theta^iv)=0
 \quad\text{for all }a\in K_0.
\]
The distinct field automorphisms $1,\sigma,\ldots,\sigma^{f-1}$ are linearly independent as maps $K_0\to K_0$. Hence $\lambda(v)=0$ for every $v\in C$, so the span is all of $C$.
\end{proof}

\begin{theorem}[Centralizers and semilinear descent]\label{thm:descent}
For any $D,E\in\Isoc(k)$, put $\phi_D=F_D^f$, $\phi_E=F_E^f$, and
\[
 C_\phi(E,D):=\{h\in\Hom_{K_0}(E,D):\phi_Dh=h\phi_E\}.
\]
The operator $\Theta(h):=F_DhF_E^{-1}$ preserves $C_\phi(E,D)$, and there is a canonical isomorphism
\begin{equation}\label{eq:descent-general}
 \Hom_{\Isoc(k)}(E,D)\otimes_{\Q_p}K_0
 \xrightarrow{\sim}C_\phi(E,D).
\end{equation}
In particular, for $1$-motives,
\begin{align}
 \Hom_k(M,N)\otimes\Q_p
 &\cong C_\phi(D(N)_{K_0},D(M)_{K_0})^{\Theta=1},\label{eq:descent-rational}\\
 \Hom_k(M,N)\otimes K_0
 &\cong C_\phi(D(N)_{K_0},D(M)_{K_0}).\label{eq:descent-linear}
\end{align}
\end{theorem}
\begin{proof}
The composite defining $\Theta(h)$ is $K_0$-linear, while $\Theta$ as an operator on the Hom space is $\sigma$-semilinear. The commutation with $\phi_D,\phi_E$ shows that $\Theta$ preserves $C_\phi$. On this space,
\[
 \Theta^f(h)=\phi_Dh\phi_E^{-1}=h.
\]
Its fixed vectors are exactly the maps commuting with $F$. Hence, $C_{\phi}(E,D)^{\Theta=1}=\Hom_{\Isoc(k)}(E,D)$. Lemma~\ref{lem:descent} proves \eqref{eq:descent-general}. Proposition~\ref{prop:rational} gives the last two identities.
\end{proof}

\begin{remark}
Commuting with $F^f$ does not imply commuting with $F$ for an individual $K_0$-linear map. For the unit isocrystal $(K_0,\sigma)$, the first condition permits every scalar in $K_0$, whereas the second permits precisely $\Q_p$. isomorphism~\eqref{eq:descent-linear} holds because the entire morphism space is extended from $\Q_p$ to $K_0$. No tensor product of an $F$-equivariant Hom group over $W$ is used.
\end{remark}

\begin{corollary}[Morphisms from Frobenius polynomials]\label{cor:hom-rank}
Factor the integral Frobenius polynomials \eqref{Frobenius-integral-polynomial} over $\Q$ as
\[
 P_M=\prod_R R^{m_R},\qquad P_N=\prod_R R^{n_R},
\]
where $R$ runs through the distinct monic irreducible factors over $\Q$ and occur in at least one of
\(P_M\) and \(P_N\), with zero exponents allowed. Then
\begin{equation}\label{eq:hom-rank}
 \rk_\Z\Hom_k(M,N)=\sum_R\deg(R)m_Rn_R.
\end{equation}
In particular, $\dim_{\Q}\End_\A(M)=\sum_R\deg(R)m_R^2$.
\end{corollary}
\begin{proof}
By Theorem~\ref{thm:descent},
\[
\begin{aligned}
\rk_{\Z}\Hom_k(M,N)
=\dim_{\Q}\Hom_{\A}(M,N)
=\dim_{K_0}C_\phi(D(N)_{K_0},D(M)_{K_0}),
\end{aligned}
\]
We first note that \(\phi_M\) ($\phi_N$, resp.) on isocrystal $D(M)_{K_0}$ ($D(N)_{K_0}$, resp.) is semisimple. Indeed,
Proposition~\ref{prop:splitting} reduces this assertion to the pure
weight pieces. On the weight-\(0\) piece, the lattice becomes constant over some
finite extension \(\F_{q^r}/k\), so that \(\phi\) has finite order.
On the weight-\(2\) piece, the torus becomes split over some finite
extension; over this extension, the \(q^s\)-power
Frobenius acts as multiplication by \(q^s\). Hence $q^{-1}\phi$ has finite order. Since \(K_0\) has characteristic zero, every finite-order \(K_0\)-linear operator is semisimple: its minimal polynomial divides \(T^n-1\), which is separable. It follows that \(\phi\) is semisimple on both the weight-\(0\) and weight-\(2\) pieces. These operators are semisimple because the ground field has
characteristic zero. On the weight-\(1\) piece, semisimplicity is the
corresponding result for abelian varieties over finite fields; see
\cite[\S2, pp.~138--139]{Tate} and
\cite[Ch.~2, pp.~526--527]{Waterhouse}. Hence \(\phi_M\) and
\(\phi_N\) are semisimple. After extending scalars to an algebraic closure \(\overline{K}_0\),
we therefore have the eigenspaces decomposition
\[
D(X)_{K_0}\otimes_{K_0}\overline{K}_0
=
\bigoplus_{\lambda}D_{X,\lambda},
\qquad
D_{X,\lambda}
:=
\ker(\phi_X-\lambda), \, \text{ for }X=M,N.
\]
Since \(C_\phi(D_N,D_M)\) is the kernel of the \(K_0\)-linear map
\[
\Hom_{K_0}(D_N,D_M)\longrightarrow\Hom_{K_0}(D_N,D_M),
\qquad
h\longmapsto\phi_Mh-h\phi_N,
\]
taking the kernel commutes with extension of scalars. Moreover, an
intertwiner sends \(D_{N,\lambda}\) into \(D_{M,\lambda}\). Thus
\[
C_\phi(D_N,D_M)\otimes_{K_0}\overline{K}_0
\cong
\bigoplus_{\lambda}
\Hom_{\overline{K}_0}(D_{N,\lambda},D_{M,\lambda}).
\]
Every irreducible factor \(R\in\Q[T]\) is separable,  and therefore has \(\deg(R)\) distinct
roots in \(\overline{K}_0\). If \(\lambda\) is a root of \(R\), then
the factorizations of \(P_M\) and \(P_N\), together with
semisimplicity, give
\[
\dim_{\overline{K}_0}D_{M,\lambda}=m_R,
\qquad
\dim_{\overline{K}_0}D_{N,\lambda}=n_R.
\]
Hence each root of \(R\) contributes \(m_Rn_R\) to the dimension of
the intertwiner space, and all roots of \(R\) contribute
$\deg(R)m_Rn_R$.
Summing over the distinct irreducible factors \(R\) yields
\[
\rk_{\Z}\Hom_k(M,N)
=
\sum_R\deg(R)m_Rn_R.
\]
Taking \(N=M\) 
gives
$\dim_{\Q}\End_{\A}(M)
=
\sum_R\deg(R)m_R^2$.
\end{proof}

\subsection{Extension of the finite base field}
For a 1-motive $M$ over $k$, let \(\Sigma_M\subset\overline{\Q}^{\times}\) be the sets
of distinct roots of \(P_M\), with multiplicities
\(m_M(\alpha)\). Fixing a \(\Q\)-embedding \(\overline{\Q}\hookrightarrow\overline K_0\) identifies these roots with the eigenvalues of \(\phi_M,\phi_N\). Equalities of powers and orders of roots of unity are preserved by this embedding, so the quantities below are independent of its choice.
\begin{theorem}[Rank jumps and the common field of definition]\label{thm:field-definition}
Put $k_n:=\F_{q^n}$. Then
\begin{equation}\label{eq:rank-extension}
 \rk_\Z\Hom_{k_n}(M_{k_n},N_{k_n})
 =\sum_{\substack{\alpha\in\Sigma_M,\ \beta\in\Sigma_N\\
                 \alpha^n=\beta^n}}m_M(\alpha)m_N(\beta).
\end{equation}
Let $d(M,N)$ be the least common multiple of the orders of all roots of unity $\alpha/\beta$ , with
\(\alpha\in\Sigma_M\) and \(\beta\in\Sigma_N\). The least common multiple of an empty set is $1$. Then:
\begin{enumerate}
 \item the rank in \eqref{eq:rank-extension} depends only on $\gcd(n,d(M,N))$.
 \item every morphism $M_{\bar k}\to N_{\bar k}$ is defined over $k_{d(M,N)}$.
 \item all such morphisms are defined over $k_n$ if and only if $d(M,N)$ divides $n$.
\end{enumerate}
Only eigenvalues of the same weight can contribute to a root-of-unity ratio.
\end{theorem}
\begin{proof}
Put \(K_n=W(k_n)[1/p]\). Under base change to \(k_n\), the rational realization becomes $D(M)_{K_0}\otimes_{K_0}K_n$, and its linear Frobenius is \(\phi_M^n\otimes1\), and similarly for \(N\). Indeed, the \(fn\)-th power of the extended semilinear Frobenius acts on the original factor by \(F^{fn}\) and on the coefficient field by the identity. Its eigenvalues are therefore \(\alpha^n\), with multiplicities added when powers coincide. Theorem~\ref{thm:descent} over \(k_n\), and the eigenspace calculation in Corollary~\ref{cor:hom-rank} give \eqref{eq:rank-extension}.

The equality \(\alpha^n=\beta^n\) holds precisely when \(\alpha/\beta\) is a root of unity whose order divides \(n\). Every such order divides \(d(M,N)\), so divisibility by \(n\) is equivalent to divisibility by \(\gcd(n,d(M,N))\). This proves the first assertion. Distinct weights cannot contribute because their complex absolute values differ.

Write $H_{\mathrm{geom}}
=
\Hom_{\bar k}(M_{\bar k},N_{\bar k})$. This is a finitely generated free abelian group, by the same injection into the Hom groups of the pure constituents used in Theorem~\ref{thm:integral}, applied over \(\bar k\). Indeed, a morphism annihilated by a nonzero integer \(r\) vanishes on the lattice component, and its semi-abelian component factors through the finite group scheme \(G_N[r]\) and hence vanishes, since \(G_M\) is smooth and connected. 

Every individual geometric morphism descends to a finite extension; a finite generating set therefore descends to one common finite extension. Consequently, its rank is the maximum of the ranks in
\eqref{eq:rank-extension}. That maximum is attained at \(n=d(M,N)\) and at every multiple of \(d(M,N)\).

The subgroup of morphisms defined over \(k_n\) is saturated in \(H_{\mathrm{geom}}\). Indeed, if \(ra\) is defined over \(k_n\) for a nonzero integer \(r\), then
$r(\gamma a-a)=0$ for every $\gamma\in\Gal(\bar k/k_n)$. Torsion-freeness gives \(\gamma a=a\), and Galois descent shows that \(a\) is defined over \(k_n\). Thus equality of ranks implies equality with \(H_{\mathrm{geom}}\), proving the second assertion and one direction in the third. If \(d(M,N)\nmid n\), at least one root-of-unity ratio has order not dividing \(n\). Its positive contribution is missing from \eqref{eq:rank-extension}, so the rank is strictly smaller than the maximum. This proves necessity.
\end{proof}

\begin{remark}
The theorem determines the smallest finite extension of \(k\)
over which \emph{all} geometric morphisms between the fixed pair
are defined. A particular morphism may descend to a smaller
extension. The case \(M=N\) gives the corresponding criterion
for all geometric endomorphisms.
\end{remark}

For abelian varieties, this criterion is related to the classical
study of multiplicative relations between Frobenius eigenvalues.
Zarhin's sufficiently large extensions eliminate torsion in the
group generated by those eigenvalues
\cite[Rem.~2.6]{Zarhin}.
Theorem~\ref{thm:field-definition} uses the orders of the particular
ratios relevant to the chosen Hom group to determine its exact
common field of definition.

\section{The essential image and the isogeny criterion}\label{sec:essential-image}
We first recall that an isocrystal whose linear Frobenius is
semisimple is determined, up to isomorphism, by its Frobenius
characteristic polynomial. The recognition problem therefore
reduces to determining which polynomials occur for $1$-motives.

\begin{lemma}\label{lem:isocrystal-classification}
Let $D,E\in\Isoc(k)$, and suppose that $\phi_D=F_D^f,\phi_E=F_E^f$ are semisimple. Then $D\cong E$ in $\Isoc(k)$ if and only if these two linear operators have the same characteristic polynomial.
\end{lemma}
\begin{proof}
An isomorphism preserves the characteristic polynomial. Conversely, equality of the characteristic polynomials of semisimple linear operators gives an invertible element of $C_\phi(E,D)$. By Theorem~\ref{thm:descent}, this space is the scalar extension of $H=\Hom_{\Isoc(k)}(E,D)$.

The case $D=E=0$ is immediate. Assume that they are both nonzero. Choose a $\Q_p$-basis $h_1,\dots,h_r$ of $H$ and $K_0$-bases of $E,D$. The polynomial
\[
 \det(t_1h_1+\cdots+t_rh_r)\in K_0[t_1,\dots,t_r]
\]
is nonzero, since an invertible $K_0$-linear combination exists. A nonzero polynomial over an extension field cannot vanish at every point of $\Q_p^r$: this follows by induction on $r$ from the infiniteness of $\Q_p$. Some $\Q_p$-linear combination is therefore invertible. It commutes with $F_E$ and $F_D$, and its inverse does as well.
\end{proof}

A $q$-Weil number of weight $1$ is an algebraic integer $\pi$ such that $|\iota(\pi)|=\sqrt q$ for every complex embedding $\iota$. Let $R_\pi\in\Z[X]$ be its minimal polynomial and let $e_\pi$ be the Honda--Tate multiplicity, characterized by
 $P_{A_\pi}=R_\pi^{e_\pi}$,
where $A_\pi$ is a simple abelian variety in the isogeny class associated with $\pi$. This integer is explicit. If $E_\pi=\Q(\pi)$, then $e_\pi$ is
the least common denominator of the local invariants of the
division algebra $\End_k(A_\pi)\otimes\Q$. The invariant is
$1/2$ at each real place, zero at each complex place and at
finite places away from $p$, and at $v\mid p$ it is
\begin{equation}\label{eq:HT-invariant}
\frac{v(\pi)}{v(q)}[E_{\pi,v}:\Q_p]\pmod\Z.
\end{equation}
Here $v$ denotes an additive valuation at the indicated place; the ratio is independent of its normalization. The classification by $q$-Weil numbers and the description of $e_\pi$ through local invariants are the classical Honda--Tate theorem \cite[Ch.~2, pp.~527--528]{Waterhouse}; see the original existence result \cite{Honda}. Passing to the opposite algebra negates the local invariants and leaves $e_\pi$ unchanged.

\begin{theorem}[Recognition of whole $1$-motives]\label{thm:essential-image}
An isocrystal $D\in\Isoc(k)$ is isomorphic to $D(M)_{K_0}$ for a $1$-motive $M/k$ if and only if $\phi=F^f$ is semisimple and
\begin{equation}\label{eq:essential-polynomial}
 \det(X-\phi)=P_0(X)P_1(X)P_2(X),
\end{equation}
where
\begin{align*}
 P_0(X)&=\prod_{m\geq1}\Phi_m(X)^{a_m},\\
 P_2(X)&=\prod_{m\geq1}\bigl(q^{\deg\Phi_m}\Phi_m(X/q)\bigr)^{b_m},\\
 P_1(X)&=\prod_{\pi}R_\pi(X)^{c_\pi},\qquad e_\pi\mid c_\pi.
\end{align*}
Here $\Phi_m$ is the $m$-th cyclotomic polynomial, all exponents are nonnegative integers with indices are finite, and the last product is indexed by conjugacy classes of weight-$1$ $q$-Weil numbers. The factors $P_0,P_1,P_2$ and the associated weight filtration are uniquely determined by $D$.
\end{theorem}
\begin{proof}
For a $1$-motive, Proposition~\ref{prop:splitting} reduces the
realization to its pure constituents. As we saw in the proof of Corollary~\ref{cor:hom-rank},  Frobenius is semisimple
on each constituent. Equation~\eqref{eq:weights} gives
the three weights. A rational lattice representation of finite cyclic monodromy has a product of cyclotomic characteristic polynomials. Character duality gives the stated rescaled polynomials for tori; replacing a root of unity by its inverse does not change its cyclotomic polynomial. The decomposition of an abelian variety into powers of simple varieties and the Honda--Tate theorem give exactly the divisibilities $e_\pi\mid c_\pi$.

Conversely, construct a lattice with polynomial $P_0$ by taking $a_m$ copies of the integral representation $\Z[X]/(\Phi_m)$ on which Frobenius acts by $X$. This action has finite order and therefore defines a $k$-lattice. The corresponding construction on a character lattice produces a torus with polynomial $P_2$. The abelian variety
\[
 A=\prod_\pi A_\pi^{c_\pi/e_\pi}
\]
has polynomial $P_1$. The split $1$-motive $[L\xrightarrow0 A\times T]$ thus has the same Frobenius polynomial as $D$ and semisimple Frobenius. Lemma~\ref{lem:isocrystal-classification} supplies an isomorphism of isocrystals. 

Distinct complex absolute values separate the three factors.
Since these factors are pairwise coprime, the intrinsic
decomposition and filtration are
\[
D=\bigoplus_{i=0}^2\ker P_i(\phi),
\qquad
\wt_rD=\bigoplus_{i\leq r}\ker P_i(\phi)
\quad (0\leq r\leq2).
\]
These subspaces are $F$-stable because the $P_i$ have
$\sigma$-fixed coefficients and $F$ commutes with $\phi$.
For a realization, this filtration agrees with its weight
filtration.
\end{proof}

\begin{corollary}[Isogeny criterion]\label{cor:isogeny}
For $1$-motives $M,N/k$, the following are equivalent:
\begin{enumerate}
 \item $M$ and $N$ are isomorphic in $\A$;
 \item $P_M=P_N$;
 \item $D(M)_{K_0}\cong D(N)_{K_0}$ as $F$-isocrystals;
 \item their $K_0$-linear Frobenius realizations $(D(M)_{K_0},F_M^f)$ and $(D(N)_{K_0},F_N^f)$ are isomorphic.
\end{enumerate}
\end{corollary}
\begin{proof}
The implications $(1)\Rightarrow(3)\Rightarrow(4)
\Rightarrow(2)$ are immediate. Equality of $P_M,P_N$ separates into equality at each weight. Rational representations of a finite cyclic group are classified by their cyclotomic multiplicities; the same applies to tori through characters. Tate's isogeny criterion gives the corresponding assertion for abelian varieties \cite{Tate}. Proposition~\ref{prop:splitting} then proves
$(2)\Rightarrow(1)$. Alternatively, Lemma~\ref{lem:isocrystal-classification} gives the equivalence of the two realization conditions directly.
\end{proof}

\begin{remark}
    A related integral criterion is given by
\cite[Cor.~A.1.2]{ABV}: a morphism of
$1$-motives over an algebraically closed field is an isomorphism
if it induces isomorphisms on all prime-to-$p$ Tate realizations
and on the crystalline realization. Corollary~\ref{cor:isogeny}
characterizes isogeny classes over a finite field through the
rational crystalline realization and its Frobenius polynomial.
\end{remark}

\begin{example}[A global obstruction invisible to the slope range]\label{ex:HT-obstruction}
Over $k=\F_p$, consider $D=\Q_p^2$ with
\[
 F=\begin{pmatrix}0&p\\1&0\end{pmatrix},\qquad V=F.
\]
It has the integral Dieudonn\'e lattice $\Z_p^2$, satisfies $FV=p$, and has semisimple Frobenius with polynomial $X^2-p$. Its slopes are both $1/2$, and its eigenvalues are $p$-Weil numbers. However, it is not the realization of a $1$-motive over $\F_p$.

For $\pi=\sqrt p$, the field $\Q(\sqrt p)$ has two real places,
each with invariant $1/2$. At the unique place $v\mid p$,
formula~\eqref{eq:HT-invariant} gives
\[
\frac{v(\sqrt p)}{v(p)}
[\Q_p(\sqrt p):\Q_p]
=\frac12\cdot2=1\equiv0\pmod\Z.
\]
All remaining invariants vanish, so $e_\pi=2$.
The exponent of $X^2-p$ is $1$, which fails the divisibility
condition in Theorem~\ref{thm:essential-image}.
The square $(X^2-p)^2$ is the polynomial of the corresponding
simple abelian surface
\cite[Ch.~2]{Waterhouse}.
Thus the global multiplicity condition contains information
beyond the existence of an effective Dieudonn\'e lattice.
\end{example}

\section{Scalar extension, idempotents, and the derived realization}\label{sec:derived}
\subsection{The coefficient categories}
For $E=\Q_p$ or $K_0$, define
\[
 \A_E:=\Kar(\A\otimes_\Q E).
\]
The tensor product first extends all Hom groups and keeps the same objects; $\Kar$ then splits all idempotents. The distinction matters because a rationally simple object may acquire nontrivial projectors after scalar extension.

$\Kar$ is standard notation for the Karoubi envelope, or idempotent completion of a category $\mathcal C$. Note that the objects of $\Kar(\mathcal C)$ are pairs $(X,e)$, where
$e\in\End_{\mathcal C}(X)$ satisfies $e^2=e$; the pair represents
the direct summand selected by $e$.

Let $\mathcal S_p$ be the full subcategory of $\Isoc(k)$ consisting of direct summands of finite direct sums of the objects $D(M)_{K_0}$. Equivalently, each object in $\mathcal S_p$ is an image of an idempotent $e\in\End_{\Isoc(k)}(B)$, where $B=D(M_1)_{K_0}\oplus\dots\oplus D(M_n)_{K_0}$ for some $M_1,\dots,M_n\in\Mone(k)$.

Let $\Phi(K_0)$ be the category of finite-dimensional $K_0$-vector spaces with an invertible $K_0$-linear operator. Define $\mathcal S_0$ as the analogous full subcategory of $\Phi(K_0)$ generated by $(D(M)_{K_0},\phi_M)$. In both categories the weight filtration is recovered by the spectral decomposition and is automatically respected by morphisms. Let
\[
\mathcal R:\A^{\op}\longrightarrow\Isoc(k),
\qquad
M\longmapsto D(M)_{K_0},
\]
be the contravariant crystalline realization, with its semilinear Frobenius.
\begin{proposition}\label{prop:coefficient-equivalence}
The realization induces anti-equivalences of semisimple abelian categories
\[
 \mathcal R_p:\A_{\Q_p}^{\op}\simeq\mathcal S_p,
 \qquad \mathcal R_0:\A_{K_0}^{\op}\simeq\mathcal S_0.
\]
\end{proposition}
\begin{proof}
Proposition~\ref{prop:rational} gives full faithfulness after
extension of coefficients to $\Q_p$. Together with
Theorem~\ref{thm:descent}, it gives full faithfulness after
extension to $K_0$, with the target morphisms required to commute
with the linear Frobenius. Thus, in each case, every idempotent
on a finite direct sum of realizations comes from a unique
idempotent in the corresponding coefficient category. Splitting
these idempotents gives a fully faithful functor from $\A_E^{\op}$, for $E=\Q_p$ or $K_0$,
and the definitions of $\mathcal S_p$ and $\mathcal S_0$ make it
essentially surjective.

To prove semisimplicity, consider a morphism in $\A_E$.
Choose a finite direct sum $B$ of simple objects of $\A$ such
that its source and target are direct summands of finite direct
sums of $B$. The algebra $\End_{\A}(B)$ is finite-dimensional
and semisimple over $\Q$. Since $\Q$ has characteristic zero,
$\End_{\A}(B)\otimes_{\Q}E$ is semisimple. The category of
direct summands of finite direct sums of $B$ in $\A_E$ is
equivalent to the category of finite projective modules over
this algebra (or its opposite). It is therefore semisimple
abelian. In particular, the chosen morphism has a kernel and
cokernel in $\A_E$, and its image is a direct summand.
Thus $\A_E$ is semisimple abelian. The equivalences give the
same conclusion for $\mathcal S_p$ and $\mathcal S_0$. On the realization side a semisimple $\phi$ has semisimple centralizer, and kernels and images of morphisms are direct summands.
\end{proof}

\begin{example}[Why idempotents must be included]\label{ex:idempotents}
If $A$ is an ordinary elliptic curve, its two Frobenius roots in $\Q_p$ have slopes $0$ and $1$. The corresponding spectral idempotents split $D(A)_{K_0}$ into two rank-one $F$-isocrystals. Each summand still has a weight-$1$ Frobenius eigenvalue. A whole $1$-motive with only weight $1$ is an abelian variety and has even realization rank. Hence these rank-one summands do not come from whole $1$-motives, although they are objects of $\mathcal S_p$. The uncompleted collection of whole realizations is therefore insufficient for an abelian semisimple target.
\end{example}

\subsection{The category generated by curves}
Let
\[
 \mathcal T:=\DM^{\mathrm{eff}}_{\leq1}(k;\Q)
 =\thick\langle M(X):X/k\text{ smooth},\ \dim X\leq1\rangle
\]
inside the category of effective geometric motives with rational coefficients. We use the rational equivalence
\begin{equation}\label{eq:Tot}
 \Tot:D^b(\A)\xrightarrow{\sim}\mathcal T
\end{equation}
of Orgogozo and Barbieri-Viale--Kahn \cite{Orgogozo,BVK}. Theorem~2.1.2 of \cite{BVK} gives the stronger formulation with $p$ inverted and the \'etale topology. We fix one normalization of \eqref{eq:Tot}; the cohomological reindexing used to represent the two-term complexes does not affect the comparison of morphism groups below.

For precision, define the completed scalar extension of $\mathcal T$ through this equivalence by
\[
 \mathcal T_E:=D^b(\A_E),\quad E=\Q_p \text{ or } E=K_0.
\]
Through the chosen equivalence \eqref{eq:Tot}, the natural functor
$\mathcal T\to\mathcal T_E$ extends morphism groups by $E$.
Its image generates $\mathcal T_E$ under finite sums, shifts,
and direct summands. This description supplies a triangulated scalar extension rather than leaving that operation implicit.

Recall anti-equivalences $\mathcal R_p, \mathcal R_0$ of Proposition~\ref{prop:coefficient-equivalence}; on objects coming from \(\A\), both are built from \(\mathcal R(M)\), but they retain different Frobenius structures and use different coefficient fields. We use the same symbols for the induced functors on \(\mathcal T\): compose
with \(\Tot^{-1}\), extend to bounded complexes by reversing degrees, and
use the corresponding target derived category. Thus
\[
\mathcal R_p:\mathcal T^{\op}\longrightarrow D^b(\mathcal S_p),
\qquad
\mathcal R_0:\mathcal T^{\op}\longrightarrow D^b(\mathcal S_0).
\]
In particular, for \(C\in D^b(\A_E)\), the induced realization satisfies
\(\mathcal R_E(C[1])\cong\mathcal R_E(C)[-1]\), for \(E=\Q_p,K_0\),
where \(\mathcal R_E\) denotes \(\mathcal R_p\) or \(\mathcal R_0\),
respectively. The same symbols denote the induced equivalences on \(\mathcal T_E\) and their restrictions to \(\mathcal T\)

\begin{theorem}[Derived realization in level at most one]\label{thm:derived}
There are contravariant triangulated equivalences
\[
 \mathcal T_{\Q_p}^{\op}\simeq D^b(\mathcal S_p),
 \qquad
 \mathcal T_{K_0}^{\op}\simeq D^b(\mathcal S_0).
\]
In particular, the induced functors $\mathcal R_p,\mathcal R_0$ on $\mathcal T$ satisfy
\begin{align*}
 \Hom_{\mathcal T}(X,Y)\otimes_\Q\Q_p
 &\cong\Hom_{D^b(\mathcal S_p)}(\mathcal R_pY,\mathcal R_pX),\\
 \Hom_{\mathcal T}(X,Y)\otimes_\Q K_0
 &\cong\Hom_{D^b(\mathcal S_0)}(\mathcal R_0Y,\mathcal R_0X).
\end{align*}
\end{theorem}
\begin{proof}
Apply the anti-equivalences of
Proposition~\ref{prop:coefficient-equivalence} termwise to bounded
complexes, reverse cohomological degrees, and use the standard
sign convention for the differential. The resulting equivalences
satisfy $\mathcal R_E(C[1])\cong\mathcal R_E(C)[-1]$.

All four coefficient categories are semisimple. Hence every bounded complex is isomorphic in its derived category to the sum of its cohomology objects in their respective degrees, and every acyclic bounded complex is contractible. The equivalences on the coefficient categories extend to equivalences on the derived categories. The same decomposition, together with
$\Hom_{\A}(M,N)\otimes_{\Q}E
 \cong\Hom_{\A_E}(M,N)$, shows that scalar extension commutes
with the morphism groups of bounded derived categories.
Transporting this identification through \eqref{eq:Tot} proves
the two formulas.
\end{proof}

\subsection{Extensions in the ambient isocrystal category}
The semisimple target in Theorem~\ref{thm:derived} has a concrete relation to the ambient category. The following calculation makes the extra ambient extensions explicit.

\begin{proposition}\label{prop:ambient-ext}
For $D,E\in\Isoc(k)$, let $H=\Hom_{K_0}(E,D)$ and $\Theta(h)=F_DhF_E^{-1}$. Then
\begin{equation}\label{eq:ambient-ext}
 \Ext^1_{\Isoc(k)}(E,D)\cong H/(\Theta-1)H
\end{equation}
as $\Q_p$-vector spaces. If $F_D^f$ and $F_E^f$ are semisimple, the inclusion of fixed vectors induces a canonical isomorphism
\begin{equation}\label{eq:ambient-hom-ext}
 \Hom_{\Isoc(k)}(E,D)\xrightarrow{\sim}
 \Ext^1_{\Isoc(k)}(E,D).
\end{equation}
In this isomorphism, a morphism $h$ is sent to the extension whose Frobenius, in a vector-space splitting, is
\[
 \begin{pmatrix}F_D&hF_E\\0&F_E\end{pmatrix}.
\]
\end{proposition}
\begin{proof}
Every extension splits as a sequence of $K_0$-vector spaces. In a chosen splitting its Frobenius has the form
\[
 F_c(d,e)=(F_Dd+cF_Ee,F_Ee),\qquad c\in H.
\]
Changing the splitting by $b\in H$ changes $c$ by $\Theta(b)-b$. Conversely, any $c$ gives an invertible semilinear Frobenius and hence an extension. These operations respect Baer sum, proving \eqref{eq:ambient-ext}.

Under the semisimplicity hypothesis, $\Theta^f=\phi_Dh\phi_E^{-1}$ is semisimple as a $K_0$-linear operator on $H$: left
composition by $\phi_D$ and right composition by
$\phi_E^{-1}$ are commuting semisimple operators. Because $\Theta$ is invertible and the coefficient field has characteristic zero, semisimplicity of $\Theta^f$ implies semisimplicity of $\Theta$; since a nontrivial Jordan block of $\Theta$ would remain nontrivial after taking its $f$-th power. Hence $\Theta$ is semisimple as a $\Q_p$-linear operator. Thus
\[
 H=\ker(\Theta-1)\oplus\im(\Theta-1).
\]
The natural map from the kernel to the cokernel is therefore an isomorphism. The kernel is the isocrystal Hom group $\Hom_{\Isoc(k)}(E,D)$, proving \eqref{eq:ambient-hom-ext}.
\end{proof}

\begin{example}\label{ex:ambient-unit}
For the unit $\mathbf1=(K_0,\sigma)$,
\[
 \Ext^1_{\Isoc(k)}(\mathbf1,\mathbf1)
 =K_0/(\sigma-1)K_0\cong\Q_p.
\]
The class of $1$ is represented by
\[
 F=\begin{pmatrix}1&1\\0&1\end{pmatrix}\sigma,
 \qquad F^f=\begin{pmatrix}1&f\\0&1\end{pmatrix}.
\]
It is nonsplit: a splitting would solve $\sigma(b)-b=1$ up to sign, contradicting the trace to $\Q_p$. Its middle term has nonsemisimple linear Frobenius and lies outside $\mathcal S_p$. Thus $D^b(\mathcal S_p)\to D^b(\Isoc(k))$ is not full in degree one. Theorem~\ref{thm:derived} specifies exactly the target in which the motivic full-faithfulness statement holds.
\end{example}

\section{Examples and geometric applications}\label{sec:examples}
\subsection{Kummer motives and a sharp integral bound}
For $a\in k^\times$, let $M_a=[\Z\xrightarrow{1\mapsto a}\Gm]$ be a Kummer motive over $k$, and put $m=\ord(a)$. A direct computation gives
\[
 \End_k(M_a)=\{(r,s)\in\Z^2:r\equiv s\pmod m\}.
\]
Since $m\mid q-1$, Theorem~\ref{thm:extension-kernel} gives an integral splitting of its $p$-divisible group, and
\[
 \End_k(M_a)\otimes\Z_p=\Z_p\times\Z_p,\qquad
 \phi_{M_a}=\begin{pmatrix}1&0\\0&q\end{pmatrix}
\]
in a suitable weight basis. Here $\Res(X-1,X-q)=1-q$ is a $p$-adic unit, in agreement with Theorem~\ref{thm:resultants}. Nonzero prime-to-$p$ extension data can therefore vanish even in the integral Barsotti--Tate realization.

\begin{example}[The denominator bound is attained]\label{ex:sharp-bound}
Over $\F_5$, take the elliptic curve
\[
 E:y^2=x^3+3x+2,\qquad P=(1,1).
\]
Its rational points are the point at infinity and
$(1,1),(1,4),(2,1),(2,4)$, so $\#E(\F_5)=5$ and $P$ has order $5$. The Frobenius polynomial is $X^2-X+5$, and $E$ is ordinary. Its discriminant is $-19$, a fundamental discriminant, so
\[
 R=\End_{\F_5}(E)=\Z[\pi],\qquad \pi^2-\pi+5=0.
\]
Indeed, $\Z[\pi]$ is already the maximal order in the ordinary endomorphism field, leaving no larger order for $R$ (see \cite[Theorem~4.2]{Waterhouse} for the description of endomorphism rings of elliptic curves). Since $\pi(P)=P$, formula~\eqref{eq:point-end-p} becomes
\[
 \End_{\F_5}(M_P)\otimes\Z_5
 =\{(z,x+y\pi):z,x,y\in\Z_5,\ z\equiv x+y\pmod5\}.
\]
This is an order of index $5$ in $\Z_5\times(R\otimes\Z_5)$. The rational lattice projector $(1,0)$ is not in this order, whereas five times that projector is. Its integral denominator is therefore exactly $5$. This agrees with
\[
 \Res(X-1,X^2-X+5)=5,
\]
and proves that the exponent in Theorem~\ref{thm:resultants} can be sharp.

The direct sum $M_P\oplus M_P$ separates the two polynomial bounds. Its weight polynomials are $(X-1)^2$ and $(X^2-X+5)^2$, so their resultant is $5^4$ and the characteristic-polynomial bound gives $d=4$. The minimal polynomials remain $X-1$ and $X^2-X+5$, with reduced resultant $5$. Thus $d_{\mathrm{red}}=1=\delta(M_P\oplus M_P)$ in Proposition~\ref{prop:reduced-resultants}, and the splitting defect is $(\Z_5/5\Z_5)^2$.
\end{example}

\subsection{Elliptic curves and quaternionic descent}
Let elliptic curve $E/k$ have Frobenius polynomial $X^2-tX+q$, and let $M=[\Z\to E]$. Since the roots of that polynomial have complex absolute value $\sqrt q>1$, by the Hasse bound for elliptic curves over finite
fields \cite[Theorem~V.1.1]{Silverman}, there are no crystalline cross-weight maps. Thus
\[
 \End_k(M)^{\op}\otimes\Q_p
 \cong\Q_p\times\bigl(\End_k(E)^{\op}\otimes\Q_p\bigr).
\]
After extending scalars to $K_0$, the complete answer is
\begin{equation}\label{eq:elliptic-centralizer}
 \End_k(M)^{\op}\otimes K_0\cong
 \begin{cases}
 K_0\times K_0[X]/(X^2-tX+q),&t^2\ne4q,\\
 K_0\times M_2(K_0),&t^2=4q.
 \end{cases}
\end{equation}
For the first case, a nonscalar $2\times2$ matrix has centralizer $K_0[\phi_E]$. For the second, semisimplicity makes $\phi_E=(t/2)\id$, so its centralizer is the full matrix algebra. The scalar case forces $q$ to be a square. No choice of Hodge line appears in this calculation.

If $E$ is ordinary, the two roots of its polynomial are in $\Q_p$, one a unit and the other of valuation $f$. Hence $\End_k(E)\otimes\Q_p\cong\Q_p\times\Q_p$. For a nonscalar supersingular Frobenius, the quadratic algebra is still the one defined by $X^2-tX+q$; it does not reduce to the scalar algebra merely because a chosen line fails to be stable. In the scalar supersingular case, the $\Q_p$-algebra is the quaternion division algebra, whose scalar extension to $K_0$ is $M_2(K_0)$.

Here is the descent explicitly. Let $E/\F_p$ have trace zero, with $p>3$, and extend the field to $\F_{p^2}$. Put $K_2=W(\F_{p^2})[1/p]$. The resulting isocrystal has a basis with
\[
 F=J\sigma,\qquad
 J=\begin{pmatrix}0&-p\\1&0\end{pmatrix},\qquad F^2=-pI.
\]
Over $\F_p$, the linear Frobenius has characteristic polynomial
$X^2+p$. It is nonscalar, so a $\Q_p$-basis puts its matrix in
the companion form $J$. After base change to $K_2$, the
semilinear Frobenius is $F=J\sigma$; hence $F^2=-pI$. A matrix commutes with $F$ precisely when it has the form
\begin{equation}\label{eq:quaternion-matrices}
 \begin{pmatrix}a&-p\sigma(c)\\c&\sigma(a)\end{pmatrix},
 \qquad a,c\in K_2.
\end{equation}
This algebra is generated by $a\mapsto\operatorname{diag}(a,\sigma(a))$ and $J$, with relations
\[
 J^2=-p,\qquad Ja=\sigma(a)J.
\]
It is a division algebra: the determinant in \eqref{eq:quaternion-matrices} is
$N_{K_2/\Q_p}(a)+pN_{K_2/\Q_p}(c)$. If both terms are nonzero, their valuations have different parity, so their sum is nonzero; the remaining cases are immediate. After extending scalars to $K_2$, Theorem~\ref{thm:descent} identifies this algebra with the full centralizer $M_2(K_2)$.

The two eigenvalues over $\F_p$ have ratio $-1$. Theorem~\ref{thm:field-definition} therefore gives $d(E,E)=2$: the endomorphism rank is $2$ over odd-degree extensions and $4$ over even-degree extensions, and every geometric endomorphism is defined over $\F_{p^2}$.

\subsection{Correspondences between curves}
For smooth projective curves, the morphism comparison has a direct interpretation by divisor correspondences.

\begin{corollary}\label{cor:curves}
Let $C,C'/k$ be smooth projective geometrically connected curves with chosen $k$-rational points. Let $J_C,J_{C'}$ be their Jacobians and $\operatorname{pr}_1,\operatorname{pr}_2$ the projections from their product. There is a natural isomorphism
\begin{align*}
 &\left(
 \frac{\operatorname{Pic}(C\times C')}
 {\operatorname{pr}_1^*\operatorname{Pic}(C)+\operatorname{pr}_2^*\operatorname{Pic}(C')}
 \right)\otimes\Q_p\\
 &\hspace{15mm}\cong
 \Hom_{\Isoc(k)}
 \bigl(H^1_{\cris}(C'/W)[1/p],H^1_{\cris}(C/W)[1/p]\bigr).
\end{align*}
Replacing $\Q_p$ by $K_0$ on the left replaces the target by the intertwiners of the linear Frobenius operators. Thus these degree-one crystalline morphisms are represented by scalar combinations of divisor correspondences.
\end{corollary}
\begin{proof}
Normalize a line bundle on $C\times C'$ by removing its restrictions along the two chosen points. The normalized bundle gives a pointed map $C\to J_{C'}$, hence a homomorphism $J_C\to J_{C'}$. Conversely, a homomorphism gives a normalized line bundle by pulling back the Poincar\'e bundle, using the canonical principal polarization of the Jacobian. These constructions are inverse, and their kernel before normalization consists precisely of the two pullback Picard groups.

The degree-one crystalline comparison identifies $D(J_C)$ with $H^1_{\cris}(C/W)$, compatibly with Frobenius and the corresponding Jacobian maps; see \cite{Oda,ABV}. Thus a homomorphism
$J_C\to J_{C'}$ induces an isocrystal morphism
$H^1_{\cris}(C'/W)[1/p]\to H^1_{\cris}(C/W)[1/p]$. Apply Proposition~\ref{prop:rational}, followed by Theorem~\ref{thm:descent} for the last assertion over $K_0$.
\end{proof}

This application uses ordinary crystalline cohomology for smooth proper curves. For a Picard $1$-motive attached to an open or singular variety, the morphism theorems apply to the $1$-motive itself. A further cohomological interpretation requires the appropriate logarithmic, rigid, or descent-defined comparison and its twist conventions. The general comparison in \cite{ABV} uses an alteration hypercovering and explicitly stated characteristic hypotheses; it cannot be replaced by an unspecified identification with ordinary crystalline cohomology of an arbitrary variety.

\subsection{The role of the finite field}
The finite-field hypothesis enters the torsion argument for the geometric extensions, the Tate theorem, and the separation of weights by arithmetic Frobenius. Barsotti--Tate crystal constructions rely on the Dieudonn\'e theory of $p$-divisible groups over perfect fields
and on Witt vectors. Over an algebraically closed field, the isocrystal alone need not retain the necessary arithmetic information.

For a concrete failure, ordinary elliptic curves over $\F_5$ with traces $1$ and $2$ exist by the ordinary case of Honda--Tate theory. Their geometric endomorphism algebras are respectively $\Q(\sqrt{-19})$ and $\Q(i)$, so the curves are not isogenous over $\overline{\F}_5$. Their geometric Hom group is zero. Nevertheless, over $\overline{\F}_5$ both $p$-divisible groups are isomorphic to
 $\mu_{p^\infty}\oplus\Q_p/\Z_p$, by the classification of ordinary \(p\)-divisible groups
\cite[Chapter~IV]{Demazure}.
Thus their Dieudonn\'e realizations are isomorphic. This example compares distinct isogeny classes and exhibits exactly the arithmetic information lost when the finite-field Frobenius descent is removed.

The results above suggest several further questions. Can one determine the
full elementary-divisor structure of the splitting defect $\Delta(M)$,
and use it to describe the integral orders inside the rational
centralizer algebra? The exponent of $\Delta(M)$ is determined by
Theorem~\ref{thm:exact-denominators}, but it does not specify its
individual invariant factors. A second question is to characterize
which Frobenius-stable integral lattices arise from $1$-motives. Such
an integral recognition criterion would refine the rational
essential-image theorem; it would need to account for the geometric
extension data, of which Theorem~\ref{thm:extension-kernel} detects
the two adjacent classes.

\begin{thebibliography}{99}
\raggedright

\bibitem[AB11]{AB}
F.~Andreatta and A.~Bertapelle,
\emph{Universal extension crystals of $1$-motives and applications},
J. Pure Appl. Algebra \textbf{215} (2011), no.~8, 1919--1944.
\href{https://doi.org/10.1016/j.jpaa.2010.11.004}
{doi:10.1016/j.jpaa.2010.11.004}.

\bibitem[ABV05]{ABV}
F.~Andreatta and L.~Barbieri-Viale,
\emph{Crystalline realizations of $1$-motives},
Math. Ann. \textbf{331} (2005), no.~1, 111--172.
\href{https://doi.org/10.1007/s00208-004-0576-4}
{doi:10.1007/s00208-004-0576-4}.

\bibitem[BM19]{BM}
A.~Bertapelle and N.~Mazzari,
\emph{On deformations of $1$-motives},
Canad. Math. Bull. \textbf{62} (2019), no.~1, 11--22.
\href{https://doi.org/10.4153/CMB-2017-076-2}
{doi:10.4153/CMB-2017-076-2}.

\bibitem[BM79]{BerthelotMessing1979}
P.~Berthelot and W.~Messing,
\emph{Th\'eorie de Dieudonn\'e cristalline I},
in \emph{Journ\'ees de G\'eom\'etrie Alg\'ebrique de Rennes
(Rennes, 1978), Vol.~I},
Ast\'erisque \textbf{63} (1979), 17--37.
\href{https://www.numdam.org/item/AST_1979__63__17_0/}
{Numdam}.

\bibitem[BVK16]{BVK}
L.~Barbieri-Viale and B.~Kahn,
\emph{On the derived category of $1$-motives},
Ast\'erisque \textbf{381} (2016), xi+254 pp.
\href{https://doi.org/10.24033/ast.998}
{doi:10.24033/ast.998}.

\bibitem[Del74]{Deligne}
P.~Deligne,
\emph{Th\'eorie de Hodge. III},
Publ. Math. Inst. Hautes \'Etudes Sci. \textbf{44} (1974), 5--77.
\href{https://doi.org/10.1007/BF02685881}
{doi:10.1007/BF02685881}.

\bibitem[Dem72]{Demazure}
M.~Demazure,
\emph{Lectures on $p$-Divisible Groups},
Lecture Notes in Mathematics, vol.~302,
Springer-Verlag, Berlin, 1972.
\href{https://doi.org/10.1007/BFb0060741}
{doi:10.1007/BFb0060741}.

\bibitem[SGA3]{SGA3}
M.~Demazure and A.~Grothendieck, eds.,
\emph{Sch\'emas en groupes. S\'eminaire de G\'eom\'etrie
Alg\'ebrique du Bois Marie 1962/64 (SGA 3), tome II:
Groupes de type multiplicatif, et structure des sch\'emas en
groupes g\'en\'eraux},
Lecture Notes in Mathematics, vol.~152,
Springer-Verlag, Berlin, 1970.
\href{https://doi.org/10.1007/BFb0059005}
{doi:10.1007/BFb0059005}.

\bibitem[GS06]{GilleSzamuely}
P.~Gille and T.~Szamuely,
\emph{Central Simple Algebras and Galois Cohomology},
Cambridge Studies in Advanced Mathematics, vol.~101,
Cambridge University Press, Cambridge, 2006.
\href{https://doi.org/10.1017/CBO9780511607219}
{doi:10.1017/CBO9780511607219}.

\bibitem[Hon68]{Honda}
T.~Honda,
\emph{Isogeny classes of abelian varieties over finite fields},
J. Math. Soc. Japan \textbf{20} (1968), nos.~1--2, 83--95.
\href{https://doi.org/10.2969/jmsj/02010083}
{doi:10.2969/jmsj/02010083}.

\bibitem[HL12]{HL}
E.~W.~Howe and K.~E.~Lauter,
\emph{New methods for bounding the number of points on curves
over finite fields},
in \emph{Geometry and Arithmetic}
(C.~Faber, G.~Farkas, and R.~de Jong, eds.),
European Mathematical Society, 2012, 173--212.
\href{https://doi.org/10.4171/119-1/12}
{doi:10.4171/119-1/12}.

\bibitem[MS22]{MarsegliaScholl}
S.~Marseglia and T.~Scholl,
\emph{Products and polarizations of super-isolated abelian varieties},
Math. Z. \textbf{300} (2022), no.~1, 445--462.
\href{https://doi.org/10.1007/s00209-021-02791-x}
{doi:10.1007/s00209-021-02791-x}.

\bibitem[Mes72]{Messing}
W.~Messing,
\emph{The Crystals Associated to Barsotti--Tate Groups:
With Applications to Abelian Schemes},
Lecture Notes in Mathematics, vol.~264,
Springer-Verlag, Berlin, 1972.
\href{https://doi.org/10.1007/BFb0058301}
{doi:10.1007/BFb0058301}.

\bibitem[Oda69]{Oda}
T.~Oda,
\emph{The first de Rham cohomology group and Dieudonn\'e modules},
Ann. Sci. \'Ec. Norm. Sup\'er. (4) \textbf{2} (1969),
no.~1, 63--135.
\href{https://doi.org/10.24033/asens.1175}
{doi:10.24033/asens.1175}.

\bibitem[Org04]{Orgogozo}
F.~Orgogozo,
\emph{Isomotifs de dimension inf\'erieure ou \'egale \`a un},
Manuscripta Math. \textbf{115} (2004), no.~3, 339--360.
\href{https://doi.org/10.1007/s00229-004-0495-4}
{doi:10.1007/s00229-004-0495-4}.

\bibitem[Poh91]{Pohst}
M.~Pohst,
\emph{A note on index divisors},
in \emph{Computational Number Theory (Debrecen, 1989)}
(A.~Peth\H{o}, M.~E.~Pohst, H.~C.~Williams, and
H.~G.~Zimmer, eds.),
de Gruyter, Berlin and New York, 1991, 173--182.
\href{https://doi.org/10.1515/9783110865950}
{DOI of the collected volume: 10.1515/9783110865950}.

\bibitem[Sil09]{Silverman}
J.~H.~Silverman,
\emph{The Arithmetic of Elliptic Curves},
2nd ed., Graduate Texts in Mathematics, vol.~106,
Springer, New York, 2009.
\href{https://doi.org/10.1007/978-0-387-09494-6}
{doi:10.1007/978-0-387-09494-6}.

\bibitem[Tat66]{Tate}
J.~Tate,
\emph{Endomorphisms of abelian varieties over finite fields},
Invent. Math. \textbf{2} (1966), 134--144.
\href{https://doi.org/10.1007/BF01404549}
{doi:10.1007/BF01404549}.

\bibitem[TW10]{TaixesWiese}
X.~Taix\'es i Ventosa and G.~Wiese,
\emph{Computing congruences of modular forms and Galois
representations modulo prime powers},
in \emph{Arithmetic, Geometry, Cryptography and Coding Theory
2009} (D.~Kohel and R.~Rolland, eds.),
Contemporary Mathematics, vol.~521,
American Mathematical Society, Providence, RI, 2010, 145--166.
\href{https://doi.org/10.1090/conm/521/10279}
{doi:10.1090/conm/521/10279}.

\bibitem[Wat69]{Waterhouse}
W.~C.~Waterhouse,
\emph{Abelian varieties over finite fields},
Ann. Sci. \'Ec. Norm. Sup\'er. (4) \textbf{2} (1969),
no.~4, 521--560.
\href{https://doi.org/10.24033/asens.1183}
{doi:10.24033/asens.1183}.

\bibitem[Zar15]{Zarhin}
Y.~G.~Zarhin,
\emph{Eigenvalues of Frobenius endomorphisms of abelian varieties
of low dimension},
J. Pure Appl. Algebra \textbf{219} (2015), no.~6, 2076--2098.
\href{https://doi.org/10.1016/j.jpaa.2014.07.024}
{doi:10.1016/j.jpaa.2014.07.024}.

\end{thebibliography}
\end{document}